\documentclass[twoside,11pt]{article} % For LaTeX2e

\usepackage{jmlr2e}

\usepackage{lastpage}

\PassOptionsToPackage{numbers}{natbib}

\usepackage[utf8]{inputenc} % allow utf-8 input
\usepackage[T1]{fontenc}    % use 8-bit T1 fonts
\usepackage{hyperref}       % hyperlinks
\usepackage{url}            % simple URL typesetting
\usepackage{booktabs}       % professional-quality tables
\usepackage{amsfonts}       % blackboard math symbols
\usepackage{nicefrac}       % compact symbols for 1/2, etc.
\usepackage{microtype}      % microtypography
\usepackage{todonotes}
\usepackage{setspace}
\usepackage{bbding,pifont,amsmath, mathtools, cancel}

\usepackage{amsmath,amssymb,mathrsfs}
\usepackage{times}
\usepackage{algorithm}
\usepackage[noend]{algpseudocode}
\newcommand{\LineComment}[1]{\Statex \hfill\textit{#1}}
\usepackage{xcolor}
\definecolor{darkblue}{RGB}{0,0,128}
\definecolor{darkred}{RGB}{128,0,0}
\definecolor{darkgreen}{RGB}{0,128,0}

\usepackage{tikz}
\usepackage{graphicx}

\usepackage[export]{adjustbox}
\usepackage{multirow}
\usepackage{rotating}
\usepackage{booktabs}
\usepackage{subcaption}

\usepackage{fancybox,framed}
\usepackage{todonotes}

\newcommand{\E}{E}

\newcommand{\B}{\mathcal{B}}
\newcommand{\N}{\mathbb{N}}
\newcommand{\G}{\mathcal{G}}
\newcommand{\X}{\mathcal{X}}

\newcommand{\norm}[1]{\left\lVert#1\right\rVert}

\providecommand{\abs}[1]{\left\lvert#1\right\rvert}
\providecommand{\norm}[1]{\left\lVert#1\right\rVert}

\newcommand{\cov}{\text{\upshape{cov}}}

\DeclareMathOperator*{\argmax}{\arg\!\max}

\jmlrheading{19}{2018}{1-\pageref{LastPage}}{9/17; Revised
11/18}{12/18}{17-547}{Steffen Gr\"unew\"alder and Azadeh Khaleghi}
\ShortHeadings{Approximations of the Restless Bandit Problem}{Gr\"unew\"alder and Khaleghi}

\firstpageno{1}

\begin{document}

\title{Approximations of the Restless Bandit Problem}

\author{\name Steffen Gr\"unew\"alder \email s.grunewalder@lancaster.ac.uk 
\AND
\name Azadeh Khaleghi \email a.khaleghi@lancaster.ac.uk 
\\
       \addr Department of Mathematics and Statistics \\
       Lancaster University\\
       Lancaster,  UK}
\editor{Peter Auer}

\maketitle

\begin{abstract}
The multi-armed restless bandit problem is studied in the case where  
the pay-off distributions are stationary $\varphi$-mixing. This version of the problem provides a more realistic model for most real-world applications, but cannot be optimally solved in practice, since it is known to be PSPACE-hard. 
The objective of this paper is to characterize a sub-class of the problem where {\em good} approximate solutions can be found using tractable approaches. Specifically, it is shown that under some conditions on the $\varphi$-mixing coefficients, a modified version of UCB can prove effective. 
The main challenge is that, unlike in the i.i.d. setting, the distributions of the sampled pay-offs may not have the same characteristics as those of the original bandit arms. In particular, the $\varphi$-mixing property does not necessarily carry over. This is overcome by carefully controlling the effect of a sampling policy on the pay-off distributions.
Some of the proof techniques developed in this paper can be more generally used in the context of online sampling under dependence. Proposed algorithms are accompanied with corresponding regret analysis. 
\end{abstract}

\section{Introduction}
As one of the simplest examples of 
sequential optimization under uncertainty, multi-armed bandit problems 
arise in various modern real-world applications, such as online advertisement, 
and Internet routing. 
These problems are typically studied under the assumption that 
the pay-offs are independently and identically distributed (i.i.d.), and the 
arms are independent. However, this assumption does not necessarily hold in many practical situations. 
Consider, for example, the problem of online advertisement in which the aim is to garner as many clicks 
as possible from a user. Grouping adverts into categories and associating with each category an arm, this problem turns into 
a multi-armed bandit.
There is dependence over time and across the arms since, for example, we expect a user to be more likely to select adverts that are related to her selections in the recent past. 

In this paper, we consider the multi-armed bandit problem in the case where the pay-offs are dependent  
and each arm evolves over time regardless of whether or not it is played. 
This is an instance of the so-called {\em restless} bandit problem \citep{WHIT88,GUH10,ORT14}. 
Since in this setting an optimal policy can leverage the inter-dependencies between the samples 
and switch between the arms at appropriate times, it can obtain an overall pay-off much higher than that given
by playing the {\em best arm}, i.e. the distribution with the highest expected pay-off, see Example~1 in \citep{ORT14}. 
However, finding the best such {\em switching strategy} is PSPACE-hard, even in the case where the process distributions are Markovian with known dynamics \citep*{TSI99}. 
Therefore, it is useful to consider relaxations of the problem with the aim to devise computationally tractable solutions that effectively approximate the optimal switching strategy. 
Approximations of the Markovian restless bandit problem {\em under known dynamics} have been previously considered, see, e.g. \citep{GUH10} and references therein. 
Our focus in this paper is on a more general setting, where the rewards have unknown distributions,
exhibit long-range dependencies, and have Markov chains as a special case. 

We are interested in a sub-class of the restless bandit problem, where the pay-offs have long-range dependencies. 
Since the nature of the problem calls for finite-time analysis, we further require that the dependence weakens over time, so as to make use of concentration inequalities in this setting. 
To this end, a natural approach is to assume that the pay-off distributions are stationary $\varphi$-mixing. 
The so-called $\varphi$-mixing coefficients $\varphi_n,~n \in \mathbb N$, of a sequence of random variables $\langle X_t\rangle_{t \in \mathbb N}$ measure the amount of dependence between the sub-sequences of $\langle X_t\rangle_{t \in \mathbb N}$ separated by $n$ time-steps. A process is said to be $\varphi$-mixing if this dependence vanishes with $n$. This notion is more formally defined in Section~\ref{sec:pre}. 
In Markov chains $\varphi$-mixing coefficients are closely related to mixing times. 
The mixing time of a Markov chain is a measure of how fast its distribution approaches the stationary distribution. In particular, it is defined to be the time that it takes for the distribution to be within $1/4^{\text{th}}$ of the stationary distribution as measured in total variation distance \citep{LEV08}[Sec 4.5]. A classical result by Davydov shows that the decrease in distance of the distribution of the Markov chain to the stationary distribution is controlled (up to a factor of $1/2$) by the $\varphi$-mixing coefficients of the Markov chain \citep{DOUK94}[pp.88]. While $\varphi$-mixing coefficients are related to the well-studied mixing properties of Markov chains, $\varphi$-mixing processes correspond to a wide variety of stochastic processes of which Markov chains are a special case \citep{DOUK94}. 

As discussed earlier, the optimal, yet notoriously infeasible strategy for this version of the problem is to switch between the arms. In this paper, we first address the question of {\em when} a relaxation obtained by identifying the arm with the highest stationary mean would lead to a viable approximation in this setting. For this purpose, we characterize the approximation error in terms of the amount of dependence between the pay-offs, and show that if $\varphi_1$ is small, the optimum of the relaxed problem is close to that given by the optimal switching strategy. Observe that this condition translates directly to the pay-off distributions being weakly dependent. 
Next, we address the question of {\em how} an optimistic approach can be devised to identify the best arm. To this end, we propose a UCB-type algorithm and show that it achieves logarithmic regret with respect to the highest stationary mean. Interestingly, the amount of dependence in the form of $\sum_i \varphi_i$ appears in the bound, and in the case where the pay-offs are i.i.d., we recover the regret bound of \cite{AUER02}. 

A familiar real-world example for the bandit problem considered in this paper corresponds to recommendation systems where the objective is to present personalized adverts, news-feeds or Massive Open Online Course (MOOC) material to each user. In these cases, depending on the specific application, each bandit arm could correspond to an appropriate subject category, indicating, for example, a 
genre of products or a class of topics. 
Once an arm is played, an item of the corresponding category could be selected at random and presented to the user. There are long-range dependencies between the pay-offs as reflected by the users' memory of their past observations. However, the dependence naturally decays over time, while users' short-term memory, which in turn influences the dominant mixing coefficients, can be controlled by capping the maximum frequency at which each specific recommendation is given to users. With minor modifications, this example carries over to a larger class of real-world sequential decision making problems which naturally possess a dependency structure imposed by users' memory.

Note that even this relaxed version of the problem is far from straightforward. The main challenge lies in obtaining confidence intervals around empirical estimates of the stationary means. Since Hoeffding-type concentration bounds exist for $\varphi$-mixing processes, it may be tempting to use such inequalities directly with standard UCB algorithms designed for the i.i.d. setting, to find the best arm. 
However, as we demonstrate in the paper, unlike in the i.i.d. setting, a policy in this framework may introduce strong couplings between past and future pay-offs in such a way that the distribution of the sampled sequence may not even be $\varphi$-mixing.
This is the reason why a standard UCB algorithm designed for i.i.d. settings is not suitable here, even when equipped with a concentration bound for $\varphi$-mixing processes. In fact, this oversight seems to have occurred in previous literature, specifically in the Improved-UCB-based approach of \cite{AUDIF15}. 
We refer to Section~\ref{sec:rt} for a more detailed discussion. 
We circumvent these difficulties by carefully taking {\em random times}\footnote{These correspond to random variables which determine the time at which an arm is sampled.}
into account and controlling the effect of a policy on the pay-off distributions. 
Some of our technical results can be more generally used to address the problem of online sampling under dependence.

Finally, while the study of the multi-armed bandit problem with {\em strongly} dependent pay-offs at its full generality is beyond the scope of this paper, we provide a complementary example for this regime. 
Specifically, we consider a setting where the bandit arms are governed by stationary Gaussian processes with slowly decaying covariance functions. Such high-dependence scenarios are quite common in practice. For instance, the throughput of radio channels changes slowly over time and the problem of choosing the best channel can be modeled by a bandit problem with strongly dependent pay-offs. 
The intuitive reason why in this setting it may also be possible to efficiently obtain approximately optimal solutions is that the strong dependencies can allow for the prediction of future rewards even from scarce observations. 
We give a simple switching strategy for this instance of the problem and show that it significantly outperforms a policy that aims for the best arm. Our regret bound for this algorithm directly reflects the dependence between the pay-offs: the higher the dependence the lower the regret. 
~\newline~
~\newline~
\noindent A summary of our main contributions is listed below.
\begin{itemize}
\item[i.] In an attempt to derive a computationally tractable solution for the restless bandit problem, we first identify a case where the optimal switching strategy can be approximated by playing the arm with the highest stationary mean. To this end, we show that the loss of settling for the highest stationary mean as opposed to finding the best switching strategy is controlled by the amount of inter-dependence as reflected by $\varphi_1$. This is shown in Proposition~\ref{prop:approx}.
\item[ii.] We provide a detailed example, namely Example~\ref{ex:nonMixingPolicy}, where we demonstrate the challenges in the non-i.i.d. bandit problem. 
In particular, we show that a policy in this framework may introduce strong couplings between past and future pay-offs in such a way that the resulting pay-off sequence may 
have a completely different dependency structure.
\item [iii.] We develop technical machinery to circumvent the difficulties introduced by the inter-dependence between the rewards, and allow us to control the effect of a policy on the pay-off distributions. 
Some of our derivations concerning sampling under dependence can be of independent interest. We further propose a UCB-type algorithm, namely Algorithm~\ref{alg:ucb} that deploys these tools to identify the arm with the highest stationary mean. 
\item[iv.] We provide an upper bound on the regret of Algorithm~\ref{alg:ucb} (with respect to the highest stationary mean). This is provided in Theorem~\ref{thm:regret_azal}.
The regret bound is a function of the amount of inter-dependence as reflected by the $\varphi$-mixing coefficients, and in the case where the pay-offs are i.i.d., we recover the regret bound of \cite{AUER02}[Thm. 1]. 
This result along with Proposition~\ref{prop:approx} allow us to argue that in the case where dependence is low Algorithm~\ref{alg:ucb} can be used to approximate the best switching strategy. 
\end{itemize}

The remainder of the paper is organized as follows. In Section \ref{sec:pre} we introduce preliminary notation and definitions. 
We formulate the problem in Section~\ref{sec:prob} and give our main results in Section~\ref{sec:results}. 
We conclude in Section \ref{sec:conc} with a discussion of open problems. 
\section{Preliminaries}\label{sec:pre}
We start with some useful notation before discussing basic definitions concerning stochastic processes. Since the bandit problem involves multiple arms (processes), we extend  the definition of  a $\varphi$-mixing process to what we call a jointly $\varphi$-mixing process, to be able to model the multi-armed bandit problem. Indeed, in many natural settings, the process is jointly $\varphi$-mixing. 
For example, as we demonstrate in Proposition \ref{lem:psi_mixing}, independent Markov chains are jointly $\varphi$-mixing.

\paragraph{Notation.}
Let $\mathbb{N}_+ := \{1,2,\ldots \}$ and $\overline{\mathbb{N}} := \mathbb{N} \cup \{\infty \}$ denote the set and extended set of natural numbers respectively. We introduce the abbreviation $\mathbf a_{m..n},~m,n \in \mathbb N_+, m \leq n$,  for sequences $a_m,a_{m+1},\dots,a_n$. Given a finite subset $C\subset \mathbb N_+$ and a sequence $\mathbf a$, we let $\mathbf a_C:=\{a_i: i \in C\}$ denote the set of elements of $\mathbf a$ indexed by $C$. If $X_C$ is a sequence of random variables indexed by $C \subset \mathbb N_+$, we denote by $\sigma(X_C)$ the smallest $\sigma$-algebra generated by $X_C$.
\paragraph{Notion of $\varphi$-dependence.}
Part of our results concern the so-called $\varphi$-dependence between $\sigma$-algebras defined as follows, see, e.g. \citep{DOUK94}. 
\begin{definition}\label{defn:phisig}
Consider a probability space $(\Omega, \mathcal A, P)$ and let $\mathcal U$ and $\mathcal V$ denote $\sigma$-subalgebras  
of $\mathcal A$ respectively. 
The $\varphi$-dependence between $\mathcal U$ and $\mathcal V$ is  is given by 
\begin{equation*}
\varphi(\mathcal U,\mathcal V) := \sup\{\abs{P(V) - P(V|U)} : U \in \mathcal U, P(U)>0, V \in \mathcal V \}.
\end{equation*}
\end{definition}
If $X$ and $Y$ are two random variables measurable with respect to $\mathcal A$ we simplify notation by letting $\varphi(X,Y):=\varphi(\sigma(X),\sigma(Y))$ denote the $\varphi$-dependence between their corresponding $\sigma$-algebras; distinction will be clear from the context. Similarly, if $X_A$ and $X_B$ are finite sequences of random variables, with $A,B \subset  \mathbb N_+$ their $\varphi$-dependence can be similarly defined as 
$$\varphi(X_A,X_B):=\varphi(\sigma(X_A),\sigma(X_B)).$$ 
In words, $\varphi(X_A,X_B)$ measures the maximal difference between the probability of an event $V$ and its conditional probability given an event $U$, where $U$ and $V$ are determined by random variables indexed by $A$ and $B$ respectively. 
The notion of $\varphi$-dependence carries over from probability measures to expectations. 
In particular, consider a real-valued random variable $X$ defined on some probability space $(\Omega,\mathcal{A},P)$, and denote by $\mathcal G$ some collected information in the form of a $\sigma$-subalgebra of $\mathcal A$. 
Let $E(X | \mathcal{G})$ denote Kolmogorov's conditional expectation, i.e. a $\mathcal{G}$-measurable random variable $Z$ such that $\int_B Z = \int_B X$ for all $B \in \mathcal{G}$.
As follows from Theorem~\ref{prop:prop1} below, due to \cite{bradley2007introduction}[vol. 1 pp. 124], the difference between $E(X | \mathcal{G})$ and  $E X$ is effectively upper-bounded by $\varphi(\mathcal G, \sigma(X))$. 
\begin{theorem}[\cite{bradley2007introduction}]\label{prop:prop1}
Let $(\Omega,\mathcal{A},P)$ be a probability space, let $X$ be a real-valued random variable with $\|X\|_1  < \infty$ and let $\mathcal{G}$ be some $\sigma$-subalgebra of $\mathcal{A}$. Then
\begin{equation}\label{eq:prop1:1}
2\varphi(\mathcal{G},\sigma(X)) = \sup \| E(Y | \mathcal{G}) - E(Y)\|_1 / \|Y\|_1,
\end{equation}
where the supremum is taken over all $\sigma(X)$-measurable random variables $Y$ with $\|Y\|_1 < \infty$. Furthermore, for any $B \in \mathcal{G}$ it holds that
\begin{equation}\label{eq:prop1:2}
\int_B |E(X | \mathcal{G}) - E(X)| \,dP \leq 2P(B)  \|X\|_\infty \varphi(\mathcal{G},\sigma(X)).    
\end{equation}
\end{theorem}
Observe that because $X$ is trivially $\sigma(X)$-measurable \eqref{eq:prop1:1} given in the theorem implies that 
\[
\| E(X | \mathcal{G}) - E(X)\|_1 \leq 2 \|X\|_1 \varphi(\mathcal{G},\sigma(X)).
\]

\paragraph{Stochastic Processes \& $\varphi$-mixing Properties.}  Let $(\mathcal X,\mathcal B_{\mathcal X})$ be a measurable space; we let $\mathcal X \subset [0,1]$ \footnote{More generally $\mathcal X$ can be a finite set or a closed interval $[a,b]$, for $a<b$, $a,b \in \mathbb{R}$.} and denote by $\mathcal B_{\mathcal X}^{(m)}$ the Borel $\sigma$-algebra on $\mathcal X^m,~m \in \N_+$. We denote by $\X^{\infty}$ the set of all $\X$-valued infinite sequences indexed by $\N_+$. A stochastic process can be modeled as a probability measure over the space $(\X^\infty, \B)$ where $\B$ denotes the $\sigma$-algebra on $\X^{\infty}$ generated by the cylinder sets. Associated with the stochastic process is a sequence of random variables $X_1, X_2, \ldots$, where $X_t:\mathcal{X}^\infty \rightarrow \mathcal{X}$ is the projection onto the $t$'th element, i.e. $X_t(\omega)=\omega_t$ for $\omega \in \mathcal X^{\infty}$ and $t \in \mathbb N_{+}$.   
A process $\rho$ is stationary if 
$\rho(X_{1.._m} \in B)= \rho(X_{i+1..i+m} \in B )$ 
for all Borel sets $B \in \mathcal B_{\mathcal{X}}^{(m)}$, $~i,m \in \mathbb{N}_+$. The term {\em stochastic process} refers to either the process distribution $\rho$ or the associated sequence of random variables $X_t,~t \in \mathbb N_+$; reference will be clear from the context.
\begin{definition}[Stationary $\varphi$-mixing Process]\label{defn:phi-single}
Consider a stationary stochastic process $\langle X_i \rangle_{i \in \mathbb{N}_+}$. 
Its $\varphi$-mixing coefficients are given by 
%For $u, v, n \in \mathbb N_+$ let 
%\begin{equation*}
%\varphi_n(u,v) := \sup\{\varphi(X_A,X_B) : A =1..u, B= u+n..u+n+v-1\}
%\end{equation*}
$$\varphi_n:=\sup_{u,v \in \mathbb{N}_+}\varphi(X_{1..u},X_{u+n..u+n+v-1}),~ n \in \N_+$$ 
and measure the $\varphi$-dependence between $\sigma(X_{1..u})$ and $\sigma(X_{u+n..u+n+v-1})$
\begin{align*}
 &      {{\underbracket[2pt]{X_1,\dots,X_u}}}{~}
      {\textcolor{white}{\underbracket[2pt]{{\textcolor{darkblue}{~\leftarrow\text{gap of length }n\rightarrow~}}}}}
 {~}      
      {{\underbracket[2pt]{X_{u+n},\ldots,X_{u+n+v-1}}}}
\end{align*}
The process is said to be $\varphi$-mixing if 
$\lim_{n\rightarrow \infty}\varphi_n=0.$ 
\end{definition}
When modeling a bandit problem in this paper, we are concerned with some $k \in \N_+$ stochastic processes with a joint distribution that is stationary $\varphi$-mixing. More specifically, for a fixed $k \in \N$, let $(\Omega,\mathcal{A},P)$ be a probability space where $\Omega := \Omega_1 \times \ldots \times \Omega_k$ with $\Omega_i:=\X^\infty,~i\in 1..k$,  $P$ a probability measure and $\mathcal{A}$ obtained via the cylinder sets. Let $\B_{\X}^{(m,k)}$ denote the Borel $\sigma$-algebra on $(\X^m)^k,~m \in \N_+$. In much the same way as with the single-process described above, associated with the joint process is a sequence of random variables $\langle X_{t,j}\rangle,~t\in \N_+,j \in 1..k$ where $X_{t,j}:\Omega \rightarrow \X$ is the projection on to the $t,j$-th element, i.e. $X_{t,j}(\omega)=\omega_{t,j}$ for $\omega \in \Omega,~t \in \N_+,j\in1..k$. 
The measure $P$ is stationary if 
$P(X_{1..m,1..k} \in B) = P(X_{i+1..i+m,1..k} \in B)$ for every $B \in \B_{\X}^{(m,k)}$ and $i,m \in \N_+$. As above, the term stochastic process is used interchangeably to correspond to the sequence of random variables $\langle X_{t,j}\rangle,~t\in \N_+,j \in 1..k$ or their corresponding joint measure $P$. 
\begin{definition}[Jointly Stationary $\varphi$-mixing Processes]\label{defn:phi-joint}
For a fixed $k \in \N_+$, consider a stationary process $\langle X_{t,j}\rangle,~t\in \N_+,j \in 1..k$. Its $\varphi$-mixing coefficients are given by 
$$\varphi_n:=\sup_{u,v \in \mathbb{N}_+}\varphi(X_{1..u,1..k},X_{u+n..u+n+v-1,1..k}),~ n \in \N_+$$ 
and measure the $\varphi$-dependence between $\sigma(X_{1..u,1..k})$ and $\sigma(X_{u+n..u+n+v-1,1..k})$. The process is said to be jointly $\varphi$-mixing if $\lim_{n\rightarrow \infty} \varphi_n = 0$. 
\end{definition}
%\paragraph{Jointly Stationary $\varphi$-mixing Processes.} 
%In much the same way as for a single process, we can define a measure $\varphi$-dependence $\varphi_n(u,v),~u,v,n  \in \mathbb N_+$ for the joint process as follows 
%$\varphi_n(u,v):=\{\varphi(X_{A,1..k},X_{B,1..k}): A=1..u,~B=a..a+v-1, a \geq u+n\}$ and 
%$\varphi(X_{A,1..k},X_{B,1..k})
%:= \sup\{\abs{P(V) - P(V|U)} : U \in \sigma(X_{A,1}, \ldots, X_{A,k}), P(U)>0, V \in \sigma(X_{B,1},\ldots, X_{B,k}) \}$. 
%We say that the joint process is $\varphi$-mixing if 
%$\lim_{n \rightarrow \infty}\varphi_n=0$, where 
%\begin{equation}\label{eq:phidef}
%\varphi_n:=\sup_{u,v \in \N } \varphi_n(u,v),~n  \in \mathbb N_+
%\end{equation} are 
%its $\varphi$-mixing coefficients. 
We use $\varphi_n$ to denote the $\varphi$-mixing coefficient corresponding to a single process or to a joint process given by Definitions~\ref{defn:phi-single}~and~\ref{defn:phi-joint} respectively; the notion will be apparent from the context. 
Under the assumption that the joint process is stationary $\varphi$-mixing, there could be dependence between the processes, as the mixing  requirement needs only to be fulfilled by the joint process. The assumption that the process is jointly $\varphi$-mixing is fulfilled by a variety of well-known models, including independent Markov chains. 
More generally, as follows from Proposition~\ref{lem:psi_mixing} below, if we have $k$ independent $\psi$-mixing processes then the joint process is $\varphi$-mixing. The $\psi$-mixing property is defined in the same way as $\varphi$-mixing whereby the $\varphi$-dependence given by Definition~\ref{defn:phisig} is replaced with 
$$\psi(\mathcal U,\mathcal V) := \sup\{\abs{1 - \rho(U\cap V)/(\rho(V) \rho(U))} : \!U \in \mathcal{U},~\rho(U)>0, V \in \mathcal{V},\rho(V) > 0 \},$$ and the $\psi$-mixing coefficients 
are defined in a manner analogous to Definitions~\ref{defn:phi-single}. A $\psi$-mixing process is also $\varphi$-mixing and the $\psi$-mixing coefficients upper bound $\varphi$-mixing coefficients. 
 \begin{proposition} \label{lem:psi_mixing}
Let $(\Omega,\mathcal{A},P)$ be some probability space with
$k$ mutually independent processes defined on it. If each of these processes is $\psi$-mixing then the joint process is also $\psi$-mixing and for all $i \in \mathbb{N}$,  $1+ \tilde \psi_i \leq (1+\psi_i)^k$, where $\tilde \psi_i$ are the mixing coefficients of the joint process and the $\psi_i$ are upper-bounds on the mixing coefficients of the individual processes. 
\end{proposition} 
The proof is provided in Appendix~\ref{app:app2}. 
Using the above result,  it can be shown that independent Markov chains are jointly $\varphi$-mixing. More specifically, we have the following example.
\begin{corollary}[Independent Markov chains are jointly $\varphi$-mixing.]\label{ex:indmc}
A set of $k \in \mathbb N_+$ mutually independent, stationary ergodic, finite-state Markov processes $\rho_i,~i=1..k$ give rise to a jointly stationary $\varphi$-mixing process.
\end{corollary} 
\begin{proof}
By Theorem~3.1 of \cite{bradley2005basic} each process $\rho_i,~i=1..k$ is $\psi$-mixing. Moreover, by Proposition~\ref{lem:psi_mixing} above, $k$ mutually independent $\psi$-mixing processes are jointly $\varphi$-mixing. As a result $\rho_i,~i=1..k$, are jointly stationary $\varphi$-mixing. 
\end{proof}
The significance of the above observation is that jointly $\varphi$-mixing processes have mutually independent Markov processes as special case.

\section{Problem Formulation: the Jointly $\varphi$-mixing Bandit Problem.}\label{sec:prob}
We assume in the sequel that a total of $k < \infty$ bandit arms are given, 
where for each $i \in 1..k$,  arm $i$ corresponds to a stationary process that generates a  time series of pay-offs $X_{1,i}, X_{2,i}, \ldots$ Furthermore, we assume that the joint process over the $k$ arms is $\varphi$-mixing in the sense of Definition~\ref{defn:phi-joint}, and that its sequence of mixing coefficients is summable, i.e. $\norm{\varphi}:=\sum_{i=1}^\infty \varphi_i < \infty$. 
Each process has stationary mean $\mu_i,~i=1..k$ and we denote by $\mu^*:= \max\{\mu_1,\ldots,\mu_k\}$ 
the highest stationary mean. We sometimes denote the arm with the highest stationary mean as the {\em best arm}. At every time-step $t \in \mathbb N_+$, a player chooses one of $k$ arms according to a policy $\pi_t$ and receives a reward $X_{t,\pi_t}$. 
The player's objective is to maximize the sum of the pay-offs received. 
The policy has access only to the pay-offs gained at earlier stages and to the arms it has chosen. Let $\langle \mathcal{F}_t\rangle_{t \geq 0}$ be a filtration that tracks the pay-offs obtained in the past $t$ rounds, i.e. $\mathcal{F}_0 = \{\emptyset,\Omega\}$, and $\mathcal{F}_{t} = \sigma(X_{1,\pi_1}, \ldots, X_{t,\pi_t}),~t \geq 1$.
%\begin{definition}[Policy]
A policy is a sequence of mappings $\pi_t:\Omega \rightarrow \{1,\ldots,k\}$, $t \geq 1$, each of which is measurable with respect to $\mathcal{F}_{t-1}$. 
%\end{definition} 
Note that the assumption that $\pi_t,~t\geq 1$ is measurable with respect to $\mathcal{F}_{t-1}$ is equivalent to the assumption that the policy can be written as a function of the past pay-offs and chosen arms, see, e.g. \citep{SHIR91}[Thm. 3, pp.174]. 
Let $\Pi = \{ \pi:=\langle \pi_t \rangle_{t\geq 1} : \pi_t \text{ is } \mathcal{F}_{t-1} \text{-measurable for all } t\geq 1\}$  denote the space of all possible policies. We also let $\G_t$ denote the filtration that keeps track of all the information available up to time $t$ (including unobserved pay-offs). More specifically, let $\langle \mathcal{G}_t\rangle_{t \geq 0}$, $\mathcal{G}_0 = \mathcal{N}$ and $\mathcal{G}_t = \sigma(\bigcup_{i=1}^k \bigcup_{s=1}^t \sigma(X_{s,i}) \cup \mathcal{N})$ for $t \geq 1$, where $\mathcal{N}$  is the family of measurable sets of $P$-measure  zero.
We define the maximal value that can be achieved in $n$ rounds as 
\begin{equation}\label{eq:nu}
v^*_n = \sup_{\pi \in \Pi} \sum_{t=1}^n \E X_{t,\pi_t}.
\end{equation}
The regret that builds up over $n$ rounds for any strategy $\pi$ is 
\begin{equation}\label{eq:regret}
\mathcal{R}_{\pi}(n) := v^*_n - E(X_{t,\pi_t}).
\end{equation}
To simplify notation, we may use $\mathcal R(n)$ when the policy $\pi$ is clear from the context. 
\begin{remark}
Requiring $\norm{\varphi}<\infty$, which is standard in the literature on empirical process theory involving mixing processes, does not lead to a strong assumption. This condition is already fulfilled by any stationary ergodic finite-state Markov chain, see, e.g. \citep{bradley2005basic}. In fact, as follows from Theorem~3.4 of \cite{bradley2005basic} the $\varphi$-mixing coefficient corresponding to any (not necessarily stationary) $\varphi$-mixing Markov process has an exponentially fast rate of decay, giving rise to a summable sequence of $\varphi$-mixing coefficients. Note that our focus in the sequel is on the more general case of $\varphi$-mixing (not necessarily Markov) processes. 
\end{remark}
\section{Main Results} \label{sec:results}
We consider the restless bandit problem in a setting where the reward distributions are jointly $\varphi$-mixing as formulated in Section~\ref{sec:prob}. 
Recall that, while the optimal strategy in this case is to switch between the arms, obtaining the best switching strategy is PSPACE-hard. 
We address the question of {\em when} and {\em how} a good and computationally tractable approximation of the optimal policy can be obtained in this setting. 
The former question is answered in Section~\ref{sec:approx} where we characterize (in terms of $\varphi_1$) the loss of settling for the highest stationary mean as opposed to following the best switching strategy. We show that for small $\varphi$-mixing coefficients, the optimum of this relaxed problem is close to that given by $v^*_n$. To answer the latter, we devise a UCB-type algorithm in Section~\ref{sec:ucb} to identify the arm with the highest stationary mean. 
The main challenge lies in building confidence intervals around empirical estimates of the stationary means. 
Indeed, as we demonstrate in Section~\ref{sec:rt}, unlike in the i.i.d. setting, a policy in this framework may introduce strong couplings between past and future pay-offs in such a way that the resulting pay-off sequence may not even be $\varphi$-mixing.  
As a result, a standard UCB algorithm designed for an i.i.d. setting 
is not suitable here, even when equipped with a Hoeffding-type concentration bound for $\varphi$-mixing processes. 
We circumvent these difficulties in Sections~\ref{sec:approx}~and~\ref{sec:ucb} by controlling the effect of a policy on the pay-off distributions. Part of the analysis in these two sections relies on some technical results for $\varphi$-mixing processes outlined in Appendix~\ref{sec:tech}, which may be of independent interest. 
%Moreover, some of our derivations concerning sampling sequentially dependent random variables can be of independent interest; these are provided in  Appendix~\ref{sec:MacroUpperBound}.
%
Finally, our results for the weakly dependent reward distributions are complemented in Section~\ref{sec:gauss}, where we study an example of a class of strongly dependent processes and give a simple switching strategy that significantly outperforms a policy that aims for the best arm. 
\subsection{Policies, Random Times, and the $\varphi$-mixing Property}\label{sec:rt} 
Recall that a policy $\pi_t,~t \in \mathbb N_+$ is a function which based on the past (observed) data samples one of $k$ bandit arms at time-step $t$. 
Therefore, since this decision is based on samples generated by a random process, the times at which bandit arms are played can be naturally modeled via {\em random times}. More formally, 
denote by $\tau_{i,j}: \Omega \rightarrow {\mathbb{N}}_+$ a random variable which determines the time at which the $j$-th arm is sampled for the $i$-th time, $i\in \mathbb N_+$ and $j \in 1..k$, and observe that for any $t \in \N$ the event $\{\tau_{i,j}=t\}$ is in $\mathcal F_{t-1}$. We denote by 
\[
X_{\tau_i,j}:=\sum_{t \in \mathbb N_+}\chi\{\tau_{i,j}=t\}X_{t,j}
\]
 the pay-off obtained from sampling arm $j$ at random time $\tau_{i,j}$ for any $i\in \mathbb N_+,~j\in 1..k$, where $\chi $ is the indicator function, i.e.  $\chi \{\tau_{i,j}=t\}$ is $1$ for all $\omega \in \Omega$ such that $\tau_{i,j}(\omega) =t$, and is $0$ otherwise. 

The main challenge in devising a policy for the $\varphi$-mixing bandit problem is that, depending on the policy used, the dependence structure of the pay-off sequence $X_{\tau_1,j},~X_{\tau_2,j},\ldots$, for $j \in 1..k$ may be completely different from that of $X_{1,j},~X_{2,j},\ldots$. Note that this differs from the simpler i.i.d. setting where the distribution of the $j$-th arm (with all its characteristics) carries over to that of the sampled sequence. % 
This is illustrated in Example~\ref{ex:nonMixingPolicy} below.
\begin{example} \label{ex:nonMixingPolicy}
Consider a two-armed bandit problem where the second arm 
is deterministically set to $0$, i.e. $X_{t,2} = 0,~t \in \N_+$ and the first arm has a 
process distribution described by a two state Markov chain with the initial distribution and stationary distribution over the states both being $(1/2 \enspace 1/2)^\top$ and with
the following transition matrix,
\[T = \begin{pmatrix}
1- \epsilon & \epsilon \\
\epsilon & 1-\epsilon
\end{pmatrix}, \quad \text{with some } \epsilon \in (0,1).
\]
Observe that for this process, if $\epsilon$ is small, with high probability the Markov chain stays in its current state.
Now consider a policy $\pi$, and denote by $\tau_1, \tau_2, \ldots$ the sequence of random times at which $\pi$ samples the first arm according to the following simple rule. Set $\tau_1=1$. For subsequent random times, if $X_{\tau_n,1} = X_{1,1}$ for $n \in \N_+$ then $\tau_{n+1} = \tau_n +1$. Otherwise, $\tau_{n+1}$ is set to be significantly larger than $\tau_n$ to guarantee that the distribution of $X_{\tau_{n+1},1}$ given $X_{\tau_n,1}$ is close to the stationary distribution of the Markov chain, during which time the first arm is sampled. The sequence $X_{\tau_1,1}, X_{\tau_2,1},\dots$ so generated is highly dependent on $X_{1,1}$ and is not $\varphi$-mixing. In fact, the expected pay-offs given the first observation, i.e. $\E(X_{\tau_n,1} | X_{1,1}), ~n \in \N_+$, are very different from the stationary mean $\E X_{1,1}$ if $\epsilon$ is small. In particular, $\E X_{1,1} = 0.5$ while $\E(X_{\tau_n,1} | X_{1,1} =1)$ is  at least $1/(1+2\epsilon) - 1/10$ due to Equation \eqref{eq:stat_of_cond}
on page \pageref{eq:stat_of_cond}
and, hence, for $\epsilon\leq 0.01$ we have that $\E(X_{\tau_n,1} | X_{1,1} =1) \geq 0.8$. 
\end{example} 
A more detailed treatment of the above example is given in Appendix \ref{app:ExamplenonMixingPolicy}.

Indeed, it is a policy's access to the (observed) past data which  can lead to strong couplings between past and future pay-offs in this framework. 
This point has been overlooked in the work of \cite{AUDIF15} which relies on Improved-UCB  \citep{AUER10} to identify the arm with the highest stationary mean, by eliminating potentially sub-optimal arms. The elimination process depends on the data, and the time-steps at which a particular arm is played depend on the remaining arms. Hence, random times and the policy's memory have to be carefully taken into consideration, as the process distribution of the sampled sequence is different from that of the corresponding arm. This notion has not been accounted for in their algorithm, and the confidence intervals involved correspond to the distributions of the arms, and not to those of the sampled sequences, and are therefore invalid in this non-i.i.d. setting. 
\subsection{Approximation Error}\label{sec:approx} 
We start by translating $\varphi$-mixing properties to those of expectations in order to control the difference between what a switching strategy can achieve as compared to the highest stationary mean. Prior to delving into the bandit problem we consider a single bounded real-valued stationary $\varphi$-mixing process $\langle X_t\rangle_{t \in \mathbb{N}_+}$ sampled at random times $\tau_1 < \tau_2< \ldots$, where $\tau_i,~i\in \mathbb{N}_+$, is fully defined by the past observations $X_{\tau_1}, \ldots, X_{\tau_{i-1}}$. 
We control the difference between the mean of the sampled process $\langle X_{\tau_i} \rangle_{i\in \mathbb{N}_+}$ from the stationary mean of the original process. 
%More specifically, we assume that each $\tau_i,~i\in \mathbb{N}_+$ is finite almost surely and fully defined by the past observations $X_{\tau_1}, \ldots, X_{\tau_{i-1}}$. 
The following proposition shows that the difference in the means is controlled by the $\varphi$-mixing coefficients and the increments in the stopping times; the proof is provided in Appendix~\ref{sec:MacroUpperBound}.

\begin{proposition}\label{prop:st-phil} 
Assume that $\langle X_t \rangle_{t \in \mathbb{N}_+}$ is a stationary $\varphi$-mixing process with mixing coefficients $\langle \varphi_i \rangle_{i \in \mathbb{N}_+}$ such that $E X_t=\mu,~t \in \mathbb{N}_+$, and $\sup_{t \in \mathbb{N}_+ }|X_t| \leq c$ for some $c \in [0,\infty)$. Furthermore, let
 $\tau_1, \tau_2, \ldots$ be a sequence of random times such that $\tau_i + \ell \leq \tau_{i+1}$ a.s. for some $\ell \geq 1$ and all $i \in \mathbb{N}_+$, and all $\tau_{i+1}$ are $\sigma(X_{\tau_1}, \ldots, X_{\tau_i})$- measurable with $\tau_1 \in \mathbb{N}_+$ being a fixed time. Then for any $n \in \mathbb{N}_+$ 
\[
 \Bigl|\frac{1}{n} \sum_{i=1}^n E X_{\tau_i} - \mu \Bigr| \leq 2 c  \varphi_\ell.
\]   
\end{proposition}
In other words, when using the sample mean of the sampled process as an estimate of the stationary mean then the bias of the estimator is bounded through the $\varphi_\ell$-mixing coefficient. This result has further implications, in particular, it is telling us something about the leverage a switching policy has. A switching policy selects effectively random times for each arm at which the arm is played and this result is saying that the summed pay-off it can gather  cannot be more than $2c n\varphi_\ell$ larger than the stationary mean of the arm. Now, a policy is free to play the arms at any time and we only know that $\tau_{i+1} \geq \tau_i +1$ for any random time $\tau_{i}$, i.e. $\ell=1$. This intuition underlies the following proposition.
\begin{proposition}\label{prop:approx}
Consider the jointly stationary $\varphi$-mixing bandit problem formulated in Section~\ref{sec:prob}. Let $\mu_1,\ldots,\mu_k$ be the means of the stationary distributions and let $\mu^*:= \max\{\mu_1,\ldots,\mu_k\}$. Let $\varphi_1$ be the first $\varphi$-mixing coefficient as given by Definition \ref{defn:phi-joint}.
%Recall the definition of the $\varphi$-mixing coefficient  given in Definition \ref{defn:phi-joint}. % \eqref{eq:phidef}. 
For every $n\geq 1$ we have 
\[
v_n^* - n \mu^* \leq 2 n \varphi_1.
\]
\end{proposition}
\begin{proof} 
Consider an arbitrary policy $\langle \pi_t \rangle_{t \in \mathbb{N}_+}$ and an arbitrary $t \in \mathbb{N}_+$ then 
\[
E X_{t,\pi_t} = \sum_{j=1}^k \sum_{i=1}^t E X_{t,j} \times \chi{\{\tau_{i,j} = t\}}. 
\]
Recall the definition of $\mathcal{G}_{t-1}$ in Section \ref{sec:prob} and observe that $\tau_{i,j}$ is $\mathcal{G}_{t-1}$-measurable. Hence, we have with $B=\{ \tau_{i,j} = t\}$ that
\[
E X_{t,j} \times \chi{\{\tau_{i,j} = t\}}  = E E(X_{t,j}|\mathcal{G}_{t-1}) \times \chi{\{\tau_{i,j} = t\}}
= \int_B E(X_{t,j}|\mathcal{G}_{t-1}).
\]
We can extend the $\varphi$-mixing property of the joint process  from $\sigma(X_{11},\ldots,X_{tk})$to $\mathcal{G}_t$ by applying Lemma \ref{lem:addZeroMeas} and we get
\[
\bigl|\int_B (E(X_{t,j}|\mathcal{G}_{t-1}) - E X_{1,j})  \bigr| \leq 2 \varphi_1 P(B).
\]
Since the different sets $\{\tau_{i,j} = t\}, j \in \{1, \ldots, k\}, i\in \{1,\ldots, t\}$  are disjoint 
\begin{align*}
 E X_{t,\pi_t} - \mu^* &\leq \sum_{j=1}^k \sum_{i=1}^t (E X_{t,j} \times \chi{\{\tau_{i,j} = t\}}  
- P(\tau_{i,j} = t) E X_{1,j}) \\ 
&\leq 2\varphi_1 \sum_{j=1}^k \sum_{i=1}^t P(\tau_{i,j} = t) \\
&\leq 2\varphi_1.
\end{align*}
\end{proof}
Observe that this relaxation introduces an inevitable linear component to the regret as shown by Proposition~\ref{prop:approx}. 
However, we argue that if the reward distributions are weakly dependent in the sense that $\varphi_1$ is small, we may settle for the best arm instead of following the best switching strategy.

\subsection{An Optimistic Approach}\label{sec:ucb}
In this section we propose a UCB-type algorithm to identify the arm with the highest stationary mean in a jointly $\varphi$-mixing bandit problem.  
Consider the bandit problem described in Section~\ref{sec:prob}, where we have $k$  arms each with a bounded stationary pay-off sequence such that the joint process is stationary $\varphi$-mixing. Suppose that the processes are weakly dependent in the sense that $\varphi_1\leq \epsilon$ for some small $\epsilon$. 
As discussed in Section~\ref{sec:approx}, in this case a policy to settle for the best arm can serve as a good approximation for the best switching strategy. More specifically, let $$\overline{\mathcal R}_{\pi}(n):=n\mu^*-\sum_{t=1}^n E X_{t,\pi_t}$$ denote the regret of a policy $\pi$ with respect to the arm with the highest stationary mean. 
From Proposition~\ref{prop:approx} we have 
$
\frac{1}{n}(\mathcal R_{\pi}(n)-\overline{\mathcal R}_{\pi}(n)) \leq \epsilon
$
and our objective in this section is to minimize $\overline{\mathcal R}_{\pi}$. 

Recall that in light of the arguments provided in Section~\ref{sec:rt} it is crucial to take a policy's access to past (observed) data into account when devising a strategy for the bandit problem in this framework. 
To address the challenge induced by the inter-dependent reward sequences obtained at random times, our approach relies on the following key observation. 
Suppose we obtain a sequence of $m \in \mathbb N_+$ consecutive samples $X_{\tau,i},X_{\tau+1,i},\dots ,X_{\tau+m,i}$ from arm $i \in 1..k$ starting at a random time $\tau$. 
For a long batch, i.e. large enough $m$, the average expectations $\frac{1}{m}\sum_{j=0}^{m-1} \E X_{\tau+j,i}$ become close to the stationary mean $\mu_i$. More formally we have Lemma~\ref{lemma11} below.
\begin{lemma}\label{lemma11}
For a fixed $i \in 1..k$ and $m \in \N_+$, consider the consecutive samples $X_{\tau,i},X_{\tau+1,i},$ $\dots,X_{\tau+m,i}$, where $\tau: \Omega \rightarrow {\mathbb{N}}_+$ is a random time at which the $i$-th arm is sampled. Let $\mu_i$ denote the stationary mean of arm $i$. 
We have 
$$ \abs{\mu_i-\frac{1}{m}\sum_{j=0}^{m-1} \E X_{\tau+j,i}}\leq \frac{2}{m}\norm{\varphi}.$$
\end{lemma}
\begin{proof}
For simplicity of notation we denote $X_{t,i}$ by $X_t$ and $X_{\tau,i}$ by $X_{\tau}$. 
Recall that $\G_t$ denotes the filtration that keeps track of all the information available up to time $t$ 
(including unobserved pay-offs).
Observe that $\chi\{\tau=t\}$ is $\G_{t-1}$-measurable so that the event $\{\tau=t\}$ is in $\G_{t-1}$ for all $t \in \N_+$. 
As a result, for any $t \in \mathbb N_+$ and $j \in 0..m-1$ we have 
\begin{equation}\label{eq:lemma:measurability}
E (\chi\{\tau=t\}X_{t+j}|\G_{t-1})=\chi\{\tau=t\}E (X_{t+j}|\G_{t-1}),
\end{equation}
see, e.g. \citep{SHIR91}[pp.216]. 
We obtain
\begin{align}
        \left \lvert m\mu_i-\sum_{j=0}^{m-1} \E X_{\tau+j,i}\right \rvert &= \left \lvert \sum_{t\in \N_+ }\sum_{j=0}^{m-1}E (\chi\{\tau=t\}X_{t+j})-E\chi\{\tau=t\}E X_t \right \rvert \label{azal:lemma:eq1} \\
        &= \left \lvert \sum_{t\in \N_+ }\sum_{j=0}^{m-1}EE (\chi\{\tau=t\}X_{t+j}|\G_{t-1})-E\chi\{\tau=t\}E X_t \right \rvert  \label{azal:lemma:eq2} \\
        &= \left \lvert \sum_{t\in \N_+ }\sum_{j=0}^{m-1}E(\chi\{\tau=t\}E (X_{t+j}|\G_{t-1}))-E\chi\{\tau=t\}E X_t \right\rvert \label{azal:lemma:eq3}\\
        &\leq \sum_{t\in \N_+ }\sum_{j=0}^{m-1}E\big (\chi\{\tau=t\} \lvert E (X_{t+j}|\G_{t-1})-E X_t \rvert \big)  \label{azal:lemma:eq4}\\
        &=\sum_{t\in \N_+ }\sum_{j=0}^{m-1}E\big (\chi\{\tau=t\} \lvert E (X_{t+j}|\G_{t-1})-E X_{t+j} \rvert \big)\label{azal:lemma:eq41}\\
        &=\sum_{t\in \N_+ }\sum_{j=0}^{m-1} \int_{\{\tau=t\}} \lvert E (X_{t+j}|\G_{t-1})-E X_{t+j} \rvert \nonumber \\
        &\leq \sum_{t\in \N_+ }\sum_{j=0}^{m-1} 2\varphi(\G_{t-1},\sigma(X_{t+j}))E\chi\{\tau=t\}\label{azal:lemma:eq43}\\
        &\leq \sum_{t\in \N_+ }\sum_{j=1}^{m} 2\varphi_jE\chi\{\tau=t\}\label{azal:lemma:eq5}\\
        &\leq 2\norm{\varphi} \label{azal:lemma:eq6}
\end{align}
where  \eqref{azal:lemma:eq1} and \eqref{azal:lemma:eq2} are due to stationarity and the law of total expectation respectively,  
\eqref{azal:lemma:eq3} follows from \eqref{eq:lemma:measurability}, \eqref{azal:lemma:eq41} follows from stationarity,
\eqref{azal:lemma:eq43} follows from Theorem~\ref{prop:prop1}, namely Inequality \eqref{eq:prop1:2}, and noting that  $\|X\|_\infty = 1$, \eqref{azal:lemma:eq5} follows directly from Definition~\ref{defn:phi-joint}, and \eqref{azal:lemma:eq6} follows from the definition of $\norm{\varphi}$. 
\end{proof}

\begin{algorithm}[t]
 \caption{A UCB-type Algorithm for $\varphi$-mixing bandits.}\label{alg:ucb}
\begin{algorithmic}[1]
\Require{Number $k$ of arms; Sum $\norm{\varphi}:=\sum_{i=1}^\infty \varphi_i$ of mixing coefficients\footnotemark}
\Ensure{\textbf{Initialization:} Play each arm once and update the empirical mean for each arm}
\For{$i = 1..k$}
\State $\overline{X}_i \gets X_{i}$
\State $s_i \gets 1$ \Comment{$s_i,~i=1..k$ denotes the number of times arm $i$ has been {\em selected}}
\EndFor
\Ensure{\textbf{Main Loop:}}
\For{$t=1..\infty$}
   \State \textbf{Arm Selection:} Select the arm that maximizes the following UCB
   $$j \gets \min \left(\argmax_{u \in 1..k} \overline{X}_u+\sqrt{\frac{8 \xi (\frac{1}{8}+\ln t)}{2^{s_u}}}+\frac{\norm{\varphi}}{2^{s_u-1}}\right),~\text{where}~\xi:= 1 + 8\norm{\varphi}$$
   \LineComment{$\triangleright$ The min operator is used to give precedence to the smaller arm-index in the case of a tie.}
    \State \textbf{Update:}  Play arm $j$ for $2^{s_j}$ consecutive iterations and update the empirical mean accordingly.
    \begin{align*}
     &t_j \gets t & t \gets t+2^{s_j} &\qquad \overline{X}_j \gets\frac{1}{2^{s_j}}\sum_{t'=t_j}^{t-1}X_{t',j}
     &s_j \gets s_j+1
    \end{align*}
 \EndFor
 \end{algorithmic}
\end{algorithm}
\footnotetext{Note that only the sum of the coefficients is needed, not the individual coefficients $\varphi_i$. }
Inspired by this result, we provide Algorithm~\ref{alg:ucb} which, given the number $k$ of arms and the sum $\norm{\varphi}$ of the $\varphi$-mixing coefficients, works as follows. 
First, each arm is sampled once for initialization. 
Next, from $t = k+1$ on, arms are played in batches of exponentially growing length. 
Specifically, at each round arm $j$ with the highest upper-confidence on its empirical mean is selected, and played for 
$2^{s_j}$ consecutive time-steps, where $s_j$ denotes the number of times that arm $j$ has been selected so far. 
The upper confidence bound is calculated based on a Hoeffding-type bound for $\varphi$-mixing processes given by Corollary 2.1 of \cite{RIO99}. 
The $2^{s_j}$ samples obtained by playing the selected arm are used in turn to calculate (from scratch) the arm's empirical mean. The algorithm does not require the values of the individual $\varphi$-mixing coefficients, but only their sum $\norm{\varphi}$. In fact, any upper-bound $\vartheta \geq \norm{\varphi}$ may be used, in which case $\vartheta$ would replace $\norm{\varphi}$ in the regret bound of Theorem~\ref{thm:regret_azal}. 

To analyze the regret of Algorithm~\ref{alg:ucb}, first recall that in an i.i.d. setting we trivially have $\overline{\mathcal R}(n) = n\mu^*-\mu_j\sum_{j=1}^k E T_j(n)=\sum_{j=1}^k \Delta_jE T_j(n)$ where $T_j(n)$ is the total number of times that arm $j$ is played by the algorithm in $n$ rounds and $\Delta_j:=\mu^*-\mu_j,~j\in1..k$. In our framework, this equality does not necessarily hold due to the inter-dependencies between the pay-offs. However, as shown in Proposition~\ref{lemma22} below, an analogous result {\em in the form of an  upper-bound} holds for our algorithm. 
\begin{proposition}\label{lemma22}
Consider the regret $\overline{\mathcal R}(n)$ of Algorithm~\ref{alg:ucb} after $n$ rounds of play. We have,
\begin{align*}
 \overline{\mathcal R}(n)  
  \leq \sum_{j=1}^k \Delta_jE T_j(n)+2k\left (\sum_{l=0}^{n}\varphi_{l}\right )\log n 
\end{align*} 
\end{proposition}
\begin{proof}
Denote by $\tau_{i,j}: \Omega \rightarrow {\mathbb{N}}_+$ the random time at which the $j$-th arm is sampled for the $i$-th time. Note that for any $t \in \N$ the event $\{\tau_{i,j}=t\}$ is measurable with respect to the filtration $\G_{t-1}$ that keeps track of all the information available up to time $t$. 
First note that
\begin{align}
 E (\chi\{\tau_{i,j}=t\}X_{t+l,j})&=E E (\chi\{\tau_{i,j}=t\}X_{t+l,j}|\G_{t-1}) \nonumber \\
				  &=E\left(\chi\{\tau_{i,j}=t\} E\left(X_{t+l,j}|\G_{t-1}\right)\right) \nonumber\\
				  &\geq \left(\mu_j-2\varphi_l \right) P\left (\tau_{i,j}=t  \right)\label{azal:lemma2:eq1}
\end{align}
where the second equality follows from the fact that the event $\{\tau_{i,j}=t\}$ is $\G_{t-1}$-measurable and \eqref{azal:lemma2:eq1} follows from Theorem~\ref{prop:prop1}. We have,
\begin{align*}
 \overline{\mathcal R}(n) &= n\mu^*-\sum_{t=1}^n E X_{t,\pi_t}\\
	       &= n\mu^*-\sum_{t=1}^n \sum_{j=1}^k\sum_{m=1}^{\log n}\sum_{l=0}^{\min\{2^m-1,n-t\}}E \chi \{\tau_{m,j}=t\}X_{t+l,j}\nonumber \\
	       &\leq n\mu^*-\sum_{t=1}^n \sum_{j=1}^k\sum_{m=1}^{\log n}\sum_{l=0}^{\min\{2^m-1,n-t\}} P\left (\tau_{m,j}=t  \right) \left(\mu_j-2\varphi_l \right)\nonumber\\
	       &\leq n\mu^*-\sum_{t=1}^n \sum_{j=1}^k\sum_{m=1}^{\log n}\sum_{l=0}^{\min\{2^m-1,n-t\}} P\left (\tau_{m,j}=t  \right) \mu_j+ 2k \left (\sum_{l=0}^{n}\varphi_l\right )\log n \nonumber \\
	       & = \sum_{j=1}^k \Delta_j E T_j(n)+2k \left (\sum_{l=0}^{n}\varphi_l\right )\log n
\end{align*}
where the last inequality follows from \eqref{azal:lemma2:eq1}.

\end{proof}
An upper-bound on $\overline{\mathcal R}(n)$ is given by Theorem~\ref{thm:regret_azal} below with proof provided in Appendix~\ref{app:thm:regret_azal}. Recall that $\Delta_i:=\mu^*-\mu_i$ is the difference between the highest stationary mean $\mu^*$ and the stationary mean $\mu_i$ of arm $i$.
\begin{theorem}[Regret Bound.]\label{thm:regret_azal}
For the regret $\overline{\mathcal R}(n)$ of Algorithm~\ref{alg:ucb} after $n$ rounds of play. We have,
\begin{align*}
 \overline{\mathcal R}(n)&\leq \sum_{\substack{i=1\\ \mu_i \neq \mu^*}}^k\frac{32(1+8\norm{\varphi})\ln n}{\Delta_i}+(1+2\pi^2/3)(\sum_{i=1}^k\Delta_i)+\norm{\varphi}\log n
\end{align*}
\end{theorem}
\paragraph{Remark.} Interestingly, $\norm{\varphi}$ appears in the bound of Theorem~\ref{thm:regret_azal}. Indeed, in the case where the pay-offs are i.i.d., we recover (up to some constant) the regret bound of \cite{AUER02}[Thm. 1]. 
\subsection{Strongly Dependent Reward Distributions: a Complementary Example }\label{sec:gauss}
At the other end of the extreme lie bandit problems with strongly dependent pay-off distributions. 
Our objective in this section is to give an example where a simple switching strategy can be obtained in this case to leverage the strong inter-dependencies between the samples. This approach gives a much higher overall pay-off than what would be given by 
settling for the best arm, and is computationally efficient. 
The intuition is that in many cases strong dependencies may allow for the prediction of future rewards even from scarce observations of a sample path.

We consider a class of stochastic  processes for which we can easily control the level of dependency. A natural choice is to use stationary Gaussian processes on $\mathbb{N}_+$. Recall that a Gaussian process is fully specified by its 
covariance function $\cov \!\!: \mathbb{N}_+ \times \mathbb{N}_+ \rightarrow \mathbb{R}$. Also, for any covariance function  Kolmogorov's consistency theorem guarantees the existence of a Gaussian process with this particular covariance function, see, e.g. \citep{GINE16}.
A Gaussian process is stationary if it has constant mean on $\mathbb{N}_+$ and its covariance can be written as  $\cov(s) = \cov(t,t+s)$ for all $t \in \mathbb{N}_+, s \in \mathbb{N}$.   We measure the degree of dependence of the process by means of H\"older-continuity of the covariance function. In particular,  we assume that there exists  some $c > 0$ and $\alpha \in (0,1]$ such that for all $s,t \in \mathbb{N}$, $|\cov(s) - \cov(t)| \leq c|s-t|^\alpha$ and we assume that the covariance function is non-negative. A low $c$ and $\alpha$ correspond to highly dependent processes since the covariance decreases slowly over time. A slowly decreasing covariance also implies large $\varphi$-mixing coefficients, since 
$
\abs{\cov(t)}/\cov^2(0)  \leq  2(\varphi(X_0,X_t) \varphi(X_t,X_0))^{1/2},
$
by \cite{RIO99}[Thm. 1.4]. 
Consider a $k$-armed bandit problem where each arm is distributed according to a stationary Gaussian process with stationary mean $\mu_i,~i=1..k$. For simplicity, we assume that the processes are mutually independent with the same, {\em unknown}, covariance function $\cov(\cdot)$. While $\cov(\cdot)$ is assumed unknown, we have access to an upper-bound on its rate of decay. That is, we are given constants $c$ and $\alpha$ such that $\cov(\cdot)$ is $\alpha,c$-H\"older continuous. 
We further assume that an upper bound on the stationary means of the processes is known and that $\cov(0) = 1$.
In order to obtain direct control on the regret $\mathcal{R}_{\pi}(n)$ we would need to make inference about the best possible switching strategy. Instead, we provide guarantees for the regret $\mathcal{R}_{\pi}^+(n)$ of policy $\pi$ with respect to the best policy that can choose arms in hindsight, i.e. 
$\mathcal{R}_{\pi}^+(n)  := \sum_{t=1}^n E( \max_{i \leq k} (X_{t,i} - X_{t,\pi_t})) $. Note that  $\mathcal{R}_{\pi}^+(n) \geq \mathcal R_{\pi}(n)$.
\begin{algorithm}[t] 
 \caption{An Algorithm for highly dependent arms.}\label{alg:gp}
 \begin{algorithmic}[1]
 \Require{a  bound $\Delta$ on the difference between the stationary means, $\max_{i,j\leq k} \abs{\mu_i - \mu_j} \leq \Delta$;  H\"older coefficients $\alpha,c$ such that $\abs{1 - \cov(t)} \leq c t^\alpha$ for all $t\geq 0$.}
   \State 
    \begin{equation*}\label{mstar}m^\star \gets 
\left\lceil \left(\frac{\sqrt{\pi}(\Delta+  \sqrt{2})}{2\alpha c^{3/2}}\right)^{\frac{1}{1+\alpha}}  \right\rceil.    
   \end{equation*}
\If {$m^\star  \leq k$}
$m^\star \gets k+1$
\EndIf
  \If {$\Delta  < \sqrt{8 (m^*)^\alpha}/(\sqrt{2c} (m^\star - k)^\alpha + k^\alpha))$} 
  \State
  \[m^\star \gets \min \argmax_{m \in \mathbb{N}} \{ m :  \Delta  \geq \sqrt{8 m^\alpha}/(\sqrt{2c} (m - k)^\alpha + k^\alpha))\}.
  \] 
  \EndIf
  \For {$l=0\dots \infty$}
\State \textbf{Phase I:} 
observe pay-offs $X_{l m^\star+1,1} = x_1, \ldots, X_{lm^\star+k,k} = x_k$ 
\[i^* \gets \min \argmax_{i \leq k} \{x_1, \ldots, x_k\}\]
\State \textbf{Phase II:} play arm $i^*$ for $m^\star-k$ steps. 
\EndFor
 \end{algorithmic}
\end{algorithm}

We provide a simple algorithm, namely, Algorithm \ref{alg:gp}, that exploits the dependence between the pay-offs. Starting from an exploration phase, the algorithm alternates between exploration and exploitation, denoted Phase I and Phase II respectively. 
In  Phase I it sweeps through all $k$ arms to observe the corresponding pay-offs. In Phase II it plays the arm with the highest observed pay-off for $m-k$ rounds, where $m$ is  a (large) constant given by \eqref{mstar} which reflects the degree of dependence between the samples in the processes. 
We need not estimate the stationary distributions in this algorithm, as bounds on the differences between the stationary means suffice. Indeed, these differences are of minor relevance unless they are high as compared to the dependence between the individual processes. 
We have the following regret bound whose proof is given in Appendix \ref{sec:UpperBoundGP}. 
\begin{proposition}\label{steffen:GP}
Given  $\Delta$ such that $\max_{i,j} |\mu_i -\mu_j| \leq \Delta$ and $\alpha,c$ such that $|1-\cov(t) | \leq c t^\alpha$ for all $t \geq 0$, the regret $\mathcal{R}^+(n)$ of  Algorithm \ref{alg:gp} after $n$ rounds is at most,
\[
  (n + m^\star) k(k-1) \left(\frac{\Delta + \sqrt{2}}{m^\star} +   
\frac{a_{m^\star} c^{1/2}}{8\pi (1-b_{m^\star})} 
\left(2\pi^{1/2}- (1-\Delta\sqrt{b_{m^\star}/4}) 
\exp\left( - \frac{\Delta^2 b_{m^\star}}{4}  \right) 
\right)\right),
\]
with $a_{m^\star} = 8c(m^\star)^\alpha, b_{m^\star} = c((m^\star - k)^\alpha + k^\alpha)$. 
\end{proposition}
%Note that the regret bound as stated above also holds when we replace  $m^\star$ with any other $m$. The particular choice $m^\star$ stated in the algorithm is the minimizer of a function that is easy to minimize and that is closely related to the  regret bound.  

To interpret the bound consider for simplicity $\alpha = 1$ and the case of a highly dependent process and, hence, a very small $c$. 
If we choose to play the arm with the highest stationary mean at all rounds, then  standard bounds on the normal distribution give us a bound of order
$n \exp(-(\Delta^2/4))$ on the regret. In this case, the regret of Algorithm \ref{alg:gp} is of order 
\[
 nk^2 c^{3/4}
%\Bigl( 3c^{1/4} + \sqrt{\frac{\Delta + \pi^{-1/2}}{\pi}}\Bigr).  
\]
because $b_{m^\star} = c m^\star \approx c^{1/4} \pi^{1/4} ((\Delta + \sqrt{2})/2)^{1/2}$ is insignificant for small $c$, $c^{1/2} a_{m^\star} =  8 c^{3/2} m^\star \approx 8 c^{3/4} \pi^{1/4} ((\Delta + \sqrt{2})/2)^{1/2}$ and the bracket on the right side is about $2\pi^{1/2} - 1$. The gap $\Delta$ itself is not of  high importance because in Phase I the algorithm selects the arm with highest current pay-off and the value stays stable over a long period as $c$ is small. Both regret bounds are linear in $n$ because the oracle has a significant advantage in this setting: at any given time $t$ the oracle chooses the arm with the highest pay-off in hindsight.  However, for a moderate number of arms $k^2 c^{3/4}$ is significantly smaller than $\exp(-\Delta^2/4)$ unless $\Delta$ is considerably large. The advantage of the switching algorithm  vanishes if $k$ is large compared to the smoothness of the process, because eventually the exploration phase (Phase I) will dominate and the smoothness of the arms cannot be exploited by this algorithm. This example demonstrates that large dependence in the stochastic process can be exploited to build switching algorithms that have a significant edge over algorithms that aim to select a single arm and algorithms like Algorithm \ref{alg:ucb} are outperformed  by simple switching algorithms.

\section{Outlook}\label{sec:conc}
This paper is an initial attempt to characterize special sub-classes of the restless bandit problem where good approximate solutions can be found using simple, computationally tractable approaches. 
We provide a UCB-type algorithm to approximate the optimal strategy in the case where the pay-offs are jointly stationary $\varphi$-mixing and are only weakly dependent. A natural open problem here is the derivation of a lower-bound. Moreover, while our algorithm only requires knowledge of the sum of the $\varphi$-mixing coefficients $\norm{\varphi}=\sum_i \varphi_i$ as opposed to that 
of each individual $\varphi_i$, the online estimation of the mixing coefficients can prove useful. Specifically, in light of Proposition~\ref{prop:approx}, if $\varphi_1$ is estimated from data, the algorithm can have a real-time estimate of its maximum loss with respect to the best switching strategy after $n$ rounds of play. Further, the results can be strengthened if the algorithm can adaptively estimate $\norm{\varphi}$ instead of relying on it as input.  Another interesting regime corresponds to strongly dependent pay-off distributions. 
We provide an example using stationary Gaussian Processes where a simple switching strategy can leverage the dependencies  to outperform a best arm policy. 
An open problem would be to weaken the assumptions on the process distributions and obtain results analogous to the weakly dependent case for the strongly dependent framework. 
\section*{Acknowledgments.}
We are grateful to the anonymous reviewers whose comments have greatly contributed to the readability and organization of the paper.

\appendix
\newtheorem{Alemma}{Lemma}[section]

\section{Proofs for $\varphi$-mixing Processes}
%In this section we provide the technical proofs omitted from the main body of the paper. 
\subsection{Technical Results} \label{sec:tech}
\label{app:variousLemma}
The following key lemma allows us to control $\varphi$-mixing coefficients corresponding to disjoint events. 
\begin{Alemma} \label{lem:phi_mixing_disjoint}
Let $(\Omega,\mathcal{A},P)$ be a probability space and let $\mathcal{B},\mathcal{C}$ be two $\sigma$-subalgebras of $\mathcal{A}$. If there exists a $\varphi\geq 0$ such that for all $B\in \mathcal{B}$ and $ C \in \mathcal{C}$ it holds that $\abs{P(B) P(C) - P(B\cap C)} \leq\varphi  P(C)$ then for any disjoint sequence $\langle B_n\rangle_{n\in\mathbb{N}}, B_n \in \mathcal{B}$ for all $n\in\mathbb{N}$, and any $C\in\mathcal{C}$, we have   
\[
\sum_{n=0}^\infty \abs{P(B_n)P(C) - P(B_n \cap C)} \leq 2 \varphi P(C). 
\]
\end{Alemma}
\begin{proof}
Let $c_n = P(B_n)P(C) - P(B_n \cap C)$, $I_+ = \{n \in \mathbb{N}: c_n \geq 0 \}, I_- =  \{n \in \mathbb{N}: c_n < 0 \}$. Now, since $\bigcup_{n \in I_+} B_n \in \mathcal{B}$ and  $\bigcup_{n \in I_-} B_n \in \mathcal{B}$ we have
\begin{align*}
P(C) \varphi &\geq P\Bigl(\bigcup_{n \in I_+} B_n\Bigr) P(C) - P\Bigl(\Bigl(\bigcup_{n \in I_+} B_n\Bigr) \cap C\Bigr)\\
&= \sum_{n \in I_+} \left(P(B_n)P(C) - P(B_n \cap C) \right) = \sum_{n \in I_+} \abs{P(B_n)P(C) - P(B_n \cap C)}.
\end{align*}
Similarly,
$P(C) \varphi \geq \sum_{n \in I_-} \abs{P(B_n)P(C) - P(B_n \cap C)}$.
Hence 
\begin{align*}
2 P(C) \varphi &\geq \!\! \sum_{n \in I_+} \abs{P(B_n)P(C) - P(B_n \cap C)} + \!\!\sum_{n \in I_-} \abs{P(B_n)P(C) - P(B_n \cap C)} \\
&= \sum_{n\in\mathbb{N}} \abs{P(B_n)P(C) - P(B_n \cap C)}. 
\end{align*} 
\end{proof}
\newpage
\noindent The following technical lemma is used in the proof of Proposition \ref{prop:approx}.
\begin{Alemma} \label{lem:addZeroMeas}
Given some probability space $(\Omega,\mathcal{A},P)$,  three $\sigma$-subalgebras $\mathcal{B}, \mathcal{C}, \mathcal{D} \subset \mathcal{A}$ and $\varphi >0$ such that $|P(U)P(V) - P(U\cap V)| \leq \varphi P(V)$ for all $U \in \mathcal{B}, V \in \mathcal{C}$ and $P(A) = 0$ for all $A \in \mathcal{D}$ then it holds that $|P(U)P(V) - P(U\cap V)| \leq \varphi P(V)$ for all $U \in \mathcal{B}, V \in \sigma(\mathcal{C} \cup \mathcal{D})$.
\end{Alemma}
\begin{proof}
\textbf{(i)} 
The set
\[
\mathcal{E} := \{A : A \in \mathcal{A}, \text{ there exists a } B \in \mathcal{C} \text{ such that for all } 
C \in \mathcal{A}, P(A \cap C) = P(B \cap C)\} 
\]
is a Dynkin system \citep{FREM}[136A]. To see this, note that (1) $\emptyset \in \mathcal{E}$ since $\emptyset \in \mathcal{C}$, (2) for $A \in \mathcal{E}$ take any $B \in \mathcal{C}$ and observe that $P(C \cap (\Omega\backslash A)) = P(C \backslash (C \cap A)) = P(C) - P(C \cap A) = P(C) - P(C \cap B) = P(C \cap (\Omega \backslash B))$ for every $C \in \mathcal{A}$ and $\Omega \backslash B \in \mathcal{C}$, (3) let $\langle A_n\rangle_{n \in \mathbb{N}}$ be a disjoint sequence in $\mathcal{E}$ with corresponding elements $B_n \in \mathcal C$ then $P(\bigcup_{n \in \mathbb{N}} A_n )  = \sum_{n=0}^\infty P(A_n)  = \sum_{n=0}^\infty P(B_n)$. Let $B_0' = B_0$ and iteratively  let $B'_n = B_n \backslash B'_{n-1}$ then $\langle B_n' \rangle_{n\in \mathbb{N}}$ is a disjoint sequence such that $P(B_n) = P(B'_n)$: for any $m,n \in \mathbb{N}$ we have $P(B_n \cap B_m) = 
P(A_n \cap B_m) = P(A_n \cap A_m) =0$ and, hence, $P(B_n \backslash B'_{n-1}) \geq  P(B_n) - \sum_{i=0}^{n-1} P(B_n \cap B_i) = P(B_n)$. Therefore, $ \sum_{n=0}^\infty P(B_n) = \sum_{n=0}^\infty P(B'_n)= P(\bigcup_{n\in \mathbb{N}} B'_n) = P(\bigcup_{n\in \mathbb{N}} B_n)$ and $\mathcal{E}$ is a Dynkin system.

\textbf{(ii)} Let $\mathcal D':=\{ A \cap B : A \in  \mathcal{C}, B \in \mathcal{D}\}$. Observe that $\mathcal{C} \cup \mathcal{D} \cup \mathcal D'$ is closed under intersection: if $A,B  \in \mathcal{C}$ then $A \cap B \in \mathcal{C}$ and similarly for $A, B \in \mathcal{D}$. If $A \in \mathcal{C}$ and $B \in \mathcal{D}$ then $A \cap B$ is in an element of $\mathcal D'$. If $A \in \mathcal{C}$ and $B$ is an element of $\mathcal D'$ then $B = C \cap D$ for some $C \in \mathcal{C}$ , $ D \in \mathcal{D}$ and $A \cap B = ( A \cap C) \cap D$ which is again an element of $\mathcal D'$ (and equivalently for $A \in \mathcal{D}$). Furthermore, $\mathcal{C} \cup \mathcal{D} \cup \{ A \cap B : A \in  \mathcal{C}, B \in \mathcal{D}\} \subset \mathcal{E}$. Note that $\mathcal{C}$ is a subset of $\mathcal E$, and if $B \in \mathcal{D}$ then we obtain $P(\emptyset \cap C) = 0 = P(B \cap C)$ for any $C \in \mathcal{A}$ and $\emptyset \in \mathcal{C}$. Finally, if we have $A \cap B$ where $A \in \mathcal{C}$ and $B \in \mathcal{D}$ then 
$P(A \cap B \cap C) =  0 = P(\emptyset \cap C)$ for any $C \in \mathcal{A}$. Hence, $\sigma(C \cup D)\subset \mathcal{E}$ by the monotone class theorem \citep{FREM}[136B].  

\textbf{(iii)} Now, let us consider a $U \in\mathcal{B}$ and $V \in \sigma(\mathcal{C} \cup \mathcal{D})$. Due to steps (i) and (ii) we can select a $V' \in \mathcal{C}$ such that $P(V \cap C) = P(V' \cap C)$ for all $C \in \mathcal{A}$. Then with this $V'$ we have that  
$ |P(U) P(V) - P(U \cap V)| = |P(U)P(V') - P(U \cap V')| \leq \varphi P(V') = \varphi P(V)$.
\end{proof}
Another important result in this context concerns the $\varphi$-mixing property of a process that starts at a random time. This result is needed to be able to use Hoeffding bounds for batches of observations which occur in our algorithm. To state this result concisely we first define the following two families of events
\begin{align*}
&\mathcal{U} := \Bigl\{ \{\tau_n < \infty \} \cap (X_{\tau_n,i},\ldots, X_{\tau_n+u-1,i})^{-1}[B] : B \in \mathcal{B}_{\mathcal{X}}^{(u)} \Bigr\},
\notag \\
&\mathcal{V} := \Bigl \{ \{\tau_n < \infty \} \cap (X_{\tau_n+u+l-1,i},\ldots, X_{\tau_n+u+l+v-2,i})^{-1}[B] : B \in \mathcal{B}_{\mathcal{X}}^{(u)}\Bigr\},
\end{align*}
where the notation $(X_{\tau_n,i},\ldots, X_{\tau_n+u-1,i})^{-1}[B]$ denotes the set 
\[\{\omega : \omega \in \Omega, (X_{\tau_n(\omega),i}(\omega), \ldots, X_{\tau_n(\omega)+u-1,i}(\omega)) \in B \}\]
and $\mathcal{X}$ is the space in which the random variables $X_{t,i}$ attain values in.
\begin{lemma} \label{lem:stopped_phi}
In the jointly-stationary $\varphi$-mixing bandit problem formulated in Section \ref{sec:prob}, consider an arm $i \leq k$ and an increasing sequence of starting times  $\langle \tau_{n} \rangle_{n \in \mathbb{N}_+}$ of batches in which arm $i$ is played which are almost surely finite. The following holds for all $n \in \mathbb{N}_+$ and $ 1 \leq l \leq 2^{n-1} -1$ 
\begin{align} \label{eq:MixingDoubling}
\sup\{ |P(V) - P(U \cap V)/P(U)| : &U  \in \mathcal{U}, V  \in \mathcal{V}, P(U)>0,  \notag\\
&u,v \in \mathbb{N}_+, u+v+l-1 \leq 2^{n-1} 
  \} \leq 2 \varphi_l.
\end{align}
\end{lemma}
\begin{proof}
For any $U,V$ as defined in (\ref{eq:MixingDoubling}) there exist sequences $\langle U_t \rangle_{t \in \mathbb{N}_+}$, and $\langle V_t \rangle_{t \in \mathbb{N}_+}$, where 
\begin{align*}
&U_t \in \sigma \left(\{ (X_{t,i},\ldots, X_{t+u-1,i})^{-1}[B] \cap \{\tau_n < \infty\} :  B \in \mathcal{B}_{\mathcal{X}}^{(u)}\}\right)\text{~and~}\\
&V_t \in \sigma \left(\{(X_{t+u+l-1,i},\ldots, X_{t+ u+l+v-2,i})^{-1}[B] \cap \{\tau_n <\infty\}: B \in \mathcal{B}_{\mathcal{X}}^{(v)}\}\right)
\end{align*} 
such that $U = \bigcup_{t \in \mathbb{N}_+} \{\tau_{n} = t\} \cap U_t$ and $V = \bigcup_{t \in \mathbb{N}_+} \{\tau_{n} = t\} \cap V_t$.  Observe also that due to stationarity $P(V_t) = P(V_1)$ for all $t \geq 1$. Since, $U_t \cap \{\tau_{n} = t\} \in \mathcal{G}_t$ the $\varphi$-mixing property gives us  
\[
|P(\{\tau_{n} = t\} \cap U_t \cap V_t) - P(\{\tau_{n}=t\} \cap U_t) P(V_t)| \leq \varphi_l P(\{\tau_{n}=t\} \cap U_t)
\]
and since $\tau_n < \infty$ almost surely for each $n$ we have $$|P(V) - P(V_1)| \leq \sum_{t=1}^\infty |P(\{\tau_{n} = t\} \cap V_t) - P(\tau_{n} = t) P(V_t)| \leq \varphi_l.$$ Combining these gives the result as 

\begin{align*}
|P(U)P(V) &- P(U \cap V)|\\ 
&= |P(U)P(V) - \sum_{t=1}^\infty P(\{\tau_n = t\} \cap U_t \cap V_t)| \\
&\leq |P(U)P(V) - \sum_{t=1}^\infty P(\{\tau_n = t\} \cap U_t) P(V_t)| + \sum_{t=1}^\infty \varphi_l P(\{\tau_n = t\} \cap U_t)
 \\
&\leq 2 \varphi_l P(U)  +  | P(U)P(V_1) - \sum_{t=1}^\infty P(\{\tau_n = t\} \cap U_t) P(V_1)| = 2 \varphi_l P(U).  
\end{align*}
\end{proof}

\subsection{Proof of Proposition \ref{lem:psi_mixing}}\label{app:app2}
\textit{Let $(\Omega,\mathcal{A},P)$ be some probability space with
$k$ mutually independent processes defined on it. If each of these processes is $\psi$-mixing then the joint process is also $\psi$-mixing and for all $i \in \mathbb{N}$,  $1+ \tilde \psi_i \leq (1+\psi_i)^k$, where the $\tilde \psi_i$ are the mixing coefficients of the joint process and the $\psi_i$ are upper bounds on the mixing coefficients of the individual processes. 
}
%\end{prop} 
\begin{proof} Fix some $n,a,u,v > 0$, $a \geq u+n$ and let us denote the individual processes with $\langle X_{n,i} \rangle_{n \in \mathbb{N}_+}$, $i \leq k$. Furthermore, let $A = \{1,\ldots, u\}, B = \{a,\ldots, a+v - 1\}$ and  $\mathcal{G} = \sigma(X_{A,1},\ldots, X_{A,k}), \mathcal{H} = \sigma(X_{B,1}, \ldots, X_{B,k})$. The proof structure is the following: in step (i) we show for ``simple'' events that the $\psi$-mixing property carries over to the joint process. In step (ii), we construct a new probability space in which to each of these ``simple'' events a product of events is associated. The approach is useful since more complex events can be easily approximated by  products. In step (iii) we make use of this approximation and we relate arbitrary events to unions of products for each of which the $\psi$-mixing property derived in (i) applies. 

\textbf{(i)}  Consider a set $U \in \mathcal{G}$ of the form $U = \bigcap_{i \leq k} U_i$ where $U_i \in \sigma(X_{A,i})$ for all $i \leq k$ and a set $V \in \mathcal{H}$ of the form $V = \bigcap_{i \leq k} V_i$, $V_i \in \sigma(X_{B,i})$, for all $i\leq k$. The mutual independence of the processes implies that
\begin{align*}
|P(U) P(V) - P(U \cap V)| &=  \Bigl| \prod_{i=1}^k P(U_i) P(V_i) -  \prod_{i=1}^k P(U_i \cap V_i)\Bigr| \\
&\leq P(U_1)P(V_1) \: \Bigl| \prod_{i=2}^k P(U_i) P(V_i) -  \prod_{i=2}^k P(U_i \cap V_i)\Bigr| \\
&\quad \quad + | P(U_1)P(V_1) - P(U_1 \cap V_1)| \, \prod_{i=2}^k P(U_i \cap V_i) \\
& \leq  \Bigl(\prod_{i=1}^2 P(U_i)P(V_i) \Bigr) \: \Bigl|\prod_{i=3}^k P(U_i) P(V_i) -  \prod_{i=3}^k P(U_i \cap V_i) \Bigr|\\
& \quad \quad +  P(U_1)P(V_1) \,  |P(U_2) P(V_2) - P(U_2 \cap V_2)| \, \prod_{i=3}^k P(U_i \cap V_i) \\
&\quad \quad +  \psi_n \, P(U_1)P(V_1) \, \prod_{i=2}^k P(U_i \cap V_i) \\
&\leq \psi_n \sum_{i=1}^k \Bigl(\, \prod_{j=1}^{i} P(U_j)P(V_j) \Bigr) \Bigl(\prod_{j=i+1}^{k} P(U_j \cap V_j) \Bigr) \\
&\leq \psi_n \,  P(U)P(V) \, \sum_{i=1}^k  \prod_{j=i+1}^{k} \frac{P(U_j \cap V_j)}{P(U_j)P(V_j)} \\
&\leq  \psi_n \, P(U) P(V) \,  \sum_{i=0}^{k-1} (1+\psi_n)^{k-i-1} 
=  ((1+\psi_n)^k -1) \, P(U)P(V). 
\end{align*}
\textbf{(ii)} Next, we demonstrate the $\psi$-mixing property  for $U  \in \mathcal{G}$ and  $V \in \mathcal{H}$. We use a product measure approach.  To make use of this let the index set $C = A \cup B$ and consider the independent $\sigma$-algebras $\sigma(X_{C,1}), \ldots,$ $\sigma(X_{C,k})$. Let $P_i$ be the restriction of $P$ to $\sigma(X_{C,i})$ for all $i \leq k$ and define the product space $(\Omega^k, \widehat\otimes_{i\leq k} \sigma(X_{C,i}), \mu)$, where $\mu$ is the product measure of $P_1, \ldots , P_k$ and $\widehat\otimes_{i\leq k} \sigma(X_{C,i})$ denotes the product $\sigma$-algebra of $\sigma(X_{C,1}),\ldots, \sigma(X_{C,k})$. The map $\phi: \Omega \rightarrow \Omega^k$, $\phi(\omega)(i) = \omega$ for all $i \leq k$, is inverse-measure preserving due to \cite{FREM}[272J,254Xc]. 
 In particular, $P(\bigcap_{i\leq k} U_i) = P \phi^{-1}[U_1 \times \ldots \times U_k] =  \mu(U_1 \times \ldots \times U_k)$ for $U_i \in \sigma(X_{A,i})$. The important property is the following: if $U \in \mathcal{G}$ then there exists an $F \in \widehat\otimes_{i\leq k} \sigma(X_{A,i}) \subseteq \widehat\otimes_{i\leq k} \sigma(X_{C,i})$ such that $U = \phi^{-1}[F]$. To see this consider the set
 \[
 \mathcal{S} := \{\phi^{-1}[F] :  F \in  \widehat\otimes_{i\leq k} \sigma(X_{A,i})\}.
 \]
A standard argument shows that $\mathcal{S}$ is a $\sigma$-algebra, i.e.  $\emptyset \in \mathcal{S}$, if $U \in \mathcal{S}$ then $\Omega \backslash U = \phi^{-1}[\Omega^k \backslash U] \in \mathcal{S}$ and if $\langle U_n \rangle_{n \in \mathbb{N}}$ a sequence in $\mathcal{S}$ then $\bigcup_{n \in \mathbb{N}} U_n = \phi^{-1}[\bigcup_{n \in \mathbb{N}} F_n] \in \mathcal{S}$ for suitable sets $F_n$. 
Furthermore, $\mathcal{S} \supseteq \{\bigcap_{i\leq k} U_i : U_i \in \sigma(X_{A,i}), i \leq k\}$ and the latter set is closed under intersection. Hence, the monotone class theorem tells us that 
$\mathcal{G} = \sigma\{\bigcap_{i\leq k} U_i : U_i \in \sigma(X_{A,i}), i \leq k\} \subset \mathcal{S}$ and the result follows. Similarly, we can link $V \in \mathcal{G}$ to $\widehat\otimes_{i\leq k} \sigma(X_{B,i})$.

\textbf{(iii)} We'd like to demonstrate that the $\psi$-mixing property carries over to arbitrary elements $U \in \widehat\otimes_{i\leq k} \sigma(X_{A,i})$ and $V \in \widehat\otimes_{i\leq k} \sigma(X_{B,i})$. First, observe that if $U = U_1 \times \ldots \times U_k$ and $V = V_1 \times \ldots \times V_k$, $U_i \in \sigma(X_{A,i}), V_i \in \sigma(X_{B,i})$, $i \leq k$, then using $(U_1 \times \ldots \times U_k) \cap (V_1 \times \ldots \times V_k) = (U_1\cap V_1) \times \ldots \times (U_k\cap V_k)$ with the inverse-measure preserving property of $\phi$ and (i) it follows  that
\begin{align*}
&|\mu(U_1 \times \ldots \times U_k) \mu(V_1 \times \ldots \times V_k) - \mu(U_1 \times \ldots \times U_k \cap V_1 \times \ldots \times V_k)| \\
&= \bigl|P\bigl(\textstyle\bigcap_{i\leq k} U_i) P\bigl(\bigcap_{i\leq k} V_i\bigr) - P\bigl( \bigcap_{i\leq k} U_i \cap \bigcap_{i\leq k} V_i\bigr) \bigr| \\
&\leq \textstyle   ((1+\psi_n)^k -1) \, P\bigl(\bigcap_{i\leq k} U_i\bigr)P\bigl(\bigcap_{i\leq k} V_i\bigr) \\
&=  ((1+\psi_n)^k -1)\,  \mu(U_1 \times \ldots \times U_k) \mu(V_1 \times \ldots \times V_k).
\end{align*}
The advantage of the product approach is that we can approximate the sets $U$ and $V$ with cylinders of the form $U_1 \times \ldots \times U_k$ and this allows us to carry the $\psi$-mixing property to 
$\mathcal{G}$ and $\mathcal{H}$.  As follows from \cite{FREM}[Thm. 251Ie, pp. 200, 251W] for every $\epsilon \in (0,1]$ there exist  sequences 
$U_{1,1}, \ldots, U_{m_1,1},~\ldots, U_{1,k}, \ldots, U_{m_1,k}$ and  $V_{1,1}, \ldots,  V_{m_2,1}, ~\ldots ,V_{1,k},\ldots, V_{m_2,k}$, with 
$U_{i,j}, V_{i',j} \in \sigma(X_{C,j})$ for all $i \leq m_1, i' \leq m_2, j \leq k$, with the following four properties.
\begin{align}
\mu\Bigl(\bigcup_{i\leq m_1} U_{i,1} \times \ldots \times U_{i,k} \Bigr) &\geq \sum_{i \leq m_1} \mu(U_{i,1} \times \ldots \times U_{i,k}) - \epsilon,  \label{eq:nearly_disjoint}\\
\mu\Bigl(\bigcup_{i\leq m_2} V_{i,1} \times \ldots \times V_{i,k} \Bigr) &\geq \sum_{i \leq m_2} \mu(V_{i,1} \times \ldots \times V_{i,k}) - \epsilon/m_1,\label{eq:Vseq} \\
\mu\Bigl(U \triangle \bigcup_{i\leq m_1} U_{i,1} \times \ldots \times U_{i,k} \Bigr) &\leq \epsilon, \text{~and~} \label{eq:nd1}\\
\mu\Bigl(V \triangle \bigcup_{i\leq m_2} V_{i,1} \times \ldots \times V_{i,k} \Bigr) &\leq \epsilon/m_1,\label{eq:nd2}
\end{align}
where $\triangle$ denotes the symmetric difference between sets.

From \eqref{eq:nearly_disjoint} and \eqref{eq:Vseq} we obtain
$\sum_{i \leq m_1} \mu(U_{i,1} \times \ldots \times U_{i,k}) \leq 1+ \epsilon$ and $\sum_{i \leq m_2} \mu(V_{i,1} \times \ldots \times V_{i,k}) \leq 1 + \epsilon/m_1$ respectively. 
Moreover, we have
\begin{align}\label{eq:Ui}
&\bigl|\mu(U) - \mu\bigl( \bigcup_{i\leq m_1} U_{i,1} \times \ldots \times U_{i,k} \bigr)\bigr| \leq \mu\bigl( U \triangle \bigcup_{i\leq m_1} U_{i,1} \times \ldots \times U_{i,k}\bigr) \leq \epsilon
\end{align}
where the second inequality follows from \eqref{eq:nd1}.
Similarly, from \eqref{eq:nd2} we obtain
\begin{align}\label{eq:Vi}
&\bigl|\mu(V) - \mu\bigl( \bigcup_{i\leq m_2} V_{i,1} \times \ldots \times V_{i,k} \bigr)\bigr| \leq \epsilon/m_1.
\end{align}
Furthermore, to obtain a similar result for $U \cap V$,below, we rely on the following elementary set manipulation. Consider sets $A_1,A_2,A_3,A_4$ then 
\begin{align}
&(A_1\cap A_2) \triangle (A_3 \cap A_4) = ((A_1 \cap A_2) \backslash (A_3 \cap A_4)) \cup ((A_3 \cap A_4) \backslash (A_1 \cap A_2)) \notag \\  
&=((A_1 \cap A_2) \backslash A_3) \cup  ((A_1 \cap A_2) \backslash A_4) \cup ((A_3 \cap A_4) \backslash A_1) \cup  ((A_3 \cap A_4) \backslash A_2) \notag  \\
&\subseteq  (A_1 \backslash A_3) \cup (A_2 \backslash A_4 ) \cup (A_3 \backslash A_1) \cup (A_4 \backslash A_2) \notag  \\ 
&= (A_1 \triangle A_3) \cup (A_2 \triangle A_4).\label{eq:set_man_proof}
\end{align}
We have,
\begin{align}
&\bigl|\mu(U\cap V) - \mu\bigl(\bigl ( \bigcup_{i\leq m_1} U_{i,1} \times \ldots \times U_{i,k} \bigr )\cap \bigl (\bigcup_{i\leq m_2} V_{i,1} \times \ldots \times V_{i,k}\bigr ) \bigr) \bigr|    \notag \\
&\leq \bigl|\mu\bigl( (U \cap V) \triangle \bigl(\bigl ( \bigcup_{i\leq m_1} U_{i,1} \times \ldots \times U_{i,k}\bigr ) \cap \bigl ( \bigcup_{i\leq m_2} V_{i,1} \times \ldots \times V_{i,k}\bigr)\bigr)\bigr) \bigr|   \notag \\
&\leq \mu \bigl (U \triangle \bigl( \bigcup_{i\leq m_1} U_{i,1} \times \ldots \times U_{i,k} \bigr)\bigr ) + 
 \mu \bigl (V \triangle \bigl( \bigcup_{i\leq m_2} V_{i,1} \times \ldots \times V_{i,k} \bigr)\bigr ) \label{eq:set_man}\\
  & \leq\epsilon +  \epsilon/m_1 \label{eq:1819} \\ 
  & \leq 2 \epsilon \label{eq:inter}
\end{align}
where, \eqref{eq:set_man} is due to \eqref{eq:set_man_proof} and \eqref{eq:1819} follows from \eqref{eq:nd1} and \eqref{eq:nd2}.
Now, applying \eqref{eq:Ui}, \eqref{eq:Vi} and \eqref{eq:inter} we obtain,
\begin{align}
&|\mu(U)\mu(V) - \mu(U \cap V)|   \nonumber \\
&\leq |\mu(U)\mu(V) - \mu\bigl(\bigcup_{i\leq m_1} \bigcup_{j \leq m_2}  (U_{i,1} \times \ldots \times U_{i,k}) \cap  (V_{j,1} \times \ldots \times  V_{j,k})\bigr)| + 2\epsilon \notag \\ 
&\leq 4  \epsilon + \bigl|\mu\bigl(\bigcup_{i\leq m_1} 
 U_{i,1} \times \ldots \times U_{i,k}\bigr) \mu\bigl(\bigcup_{i\leq m_2} V_{i,1} \times \ldots \times V_{i,k}\bigr)  \notag \\
&\qquad\qquad- \mu\bigl(\bigcup_{i\leq m_1} \bigcup_{j \leq m_2}  (U_{i,1} \times \ldots \times U_{i,k}) \cap  (V_{j,1} \times \ldots \times  V_{j,k})\bigr)\bigr|. \label{inequ:I}
\end{align}
Furthermore,
\begin{align}
&\bigl|\sum_{i\leq m_1} \sum_{j \leq m_2} 
\mu(U_{i,1} \times \ldots \times U_{i,k}) \mu( V_{j,1} \times \ldots \times V_{j,k}) \notag \\ 
&\qquad\qquad\qquad\qquad - \mu\bigl(\bigcup_{i\leq m_1} U_{i,1} \times \ldots \times U_{i,k}\bigr) \mu\bigl(\bigcup_{i\leq m_2} V_{i,1} \times \ldots \times V_{i,k}\bigr)\bigr|   \notag \\
&\leq  \Bigl( \mu\bigl(\bigcup_{i\leq m_1} U_{i,1} \times \ldots \times U_{i,k}\bigr) + \epsilon \Bigr) \sum_{j\leq m_2} \mu( V_{j,1} \times \ldots \times V_{j,k}) \notag \\
&\qquad\qquad\qquad\qquad - \mu\bigl(\bigcup_{i\leq m_1} U_{i,1} \times \ldots \times U_{i,k}\bigr) \mu\bigl(\bigcup_{i\leq m_2} V_{i,1} \times \ldots \times V_{i,k}\bigr) \notag \\
&\leq 2 \epsilon + \epsilon^2 \notag \\
&\leq 3 \epsilon.  \label{inequ:II}
\end{align}
Let $U_1, \ldots, U_m$ be such that    
 $\mu(\bigcup_{i\leq m} U_i) \geq \sum_{i\leq m} \mu(U_i) - \epsilon$. Then $\sum_{i\leq m} \mu \bigl( \bigcup_{j < i} U_j \cap U_i\bigr) \leq \epsilon$. This is because 
\begin{align*}
\mu\bigl(\bigcup_{i\leq m} U_i\bigr) &= \mu\bigl(\bigcup_{i\leq m} U_i \backslash \bigl( \bigcup_{j < i} U_j \cap U_i\bigr)\bigr)\\
&= \sum_{i\leq m} \mu\bigl(U_i \backslash \bigl( \bigcup_{j < i} U_j \cap U_i\bigr)\bigr) \\
&= \sum_{i\leq m} \mu(U_i) -  \sum_{i\leq m} \mu \bigl( \bigcup_{j < i} U_j \cap U_i\bigr)\bigr).
\end{align*}
Now, for any measurable set $V$ we have by the same argument that
\begin{align*}
\mu\bigl(\bigcup_{i\leq m} U_i \cap V \bigr) &= \mu\bigl(\bigcup_{i\leq m} U_i \cap V \backslash \bigl( \bigcup_{j < i} U_j \cap U_i \bigr)\bigr)\\
&= \sum_{i\leq m} \mu\bigl(U_i \cap V \backslash \bigl( \bigcup_{j < i} U_j \cap U_i \bigr)\bigr) \\
&= \sum_{i\leq m} \mu(U_i \cap V) -  \sum_{i\leq m} \mu \bigl( \bigcup_{j < i} U_j \cap U_i  \bigr)\bigr) \\
&\geq \sum_{i\leq m} \mu(U_i\cap V) - \epsilon. 
\end{align*}
 Hence, using  \eqref{eq:nearly_disjoint} and \eqref{eq:Vseq}, we can obtain
\begin{align}
&\bigl|\sum_{i\leq m_1} \sum_{j \leq m_2} 
\mu\bigl ((U_{i,1} \times \ldots \times U_{i,k}) \cap (V_{j,1} \times \ldots \times V_{j,k})\bigr) \notag \\
&\quad - \mu\bigl(\bigcup_{i\leq m_1} \bigcup_{j \leq m_2}  (U_{i,1} \times \ldots \times U_{i,k}) \cap (V_{j,1} \times \ldots \times V_{j,k})\bigr)\bigr|  \notag \\
&\leq \sum_{i\leq m_1} \sum_{j \leq m_2} 
\mu \bigl ((U_{i,1} \times \ldots \times U_{i,k}) \cap (V_{j,1} \times \ldots \times V_{j,k})\bigr) \notag \\
&\quad - \sum_{i\leq m_1} \mu\bigl(  (U_{i,1} \times \ldots \times U_{i,k}) \cap \bigl(\bigcup_{j \leq m_2} (V_{j,1} \times \ldots \times V_{j,k})\bigr) + \epsilon  \notag \\
&\leq \epsilon  + \sum_{i\leq m_1} \epsilon/m_1 \notag \\ 
&= 2 \epsilon.  \label{inequ:III}
\end{align}
We obtain an bound on $|\mu(U)\mu(V) - \mu(U\cap V)|$
by using \eqref{inequ:I}, \eqref{inequ:II} and \eqref{inequ:III} together with the result from Step (i), i.e., 
\begin{align}
|\mu(U)\mu(V) &- \mu(U\cap V)| \notag\\ 
&\leq 9 \epsilon  
+ \sum_{i\leq m_1} \sum_{j \leq m_2} | \mu(U_{i,1} \times \ldots \times U_{i,k}) \mu(V_{j,1} \times \ldots \times V_{j,k})  \notag \\
&\hspace{3cm}- 
\mu\bigl ( (U_{i,1} \times \ldots \times U_{i,k}) \cap (V_{j,1} \times \ldots \times V_{j,k})\bigr) | \notag  \\
&\leq 9 \epsilon + ((1+\psi_n)^k -1) \sum_{i\leq m_1} 
\mu(U_{i,1} \times \ldots \times U_{i,k})  \sum_{j \leq m_2} \mu(V_{j,1} \times \ldots \times V_{j,k}) \label{eq:critical} \\
&\leq ((1+\psi_n)^k -1)\, \mu(U)\mu(V) + 9 \epsilon + 5 \epsilon ((1+\psi_n)^k -1), \notag
\end{align}
where the last inequality follows from the same arguments as used in inequalities  \eqref{inequ:I} and \eqref{inequ:II}. Since $\epsilon$ is arbitrary we derive the upper bound for arbitrary elements $U \in \widehat \otimes_{i\leq k} \sigma(X_{A,i})$ and $V \in \widehat \otimes_{i\leq k} \sigma(X_{B,i})$.

Note that this last step is the only part of the proof where  $\psi$-mixing is necessary. In particular, we could not derive \eqref{eq:critical} from the line preceding it by only assuming $\varphi$-mixing.

 To conclude,  if $U \in \mathcal{G}, V \in \mathcal{H}$ then we know from step (ii) that there are elements $E \in \widehat \otimes_{i\leq k} \sigma(X_{A,i}), F \in \widehat \otimes_{i\leq k} \sigma(X_{B,i})$ such that 
\begin{align*}
|P(U \cap V) &-P(U) P(V)| \\
&= |P \phi^{-1}[E \cap F] - P \phi^{-1}[E] P\phi^{-1}[F]| \\
 &= |\mu(E\cap F) - \mu(E) \mu(F)| \\ 
 &\leq ((1 + \psi_n)^k -1) P(U)P(V).
\end{align*}
\textbf{(iv)} The joint process is $\psi$-mixing since 
$\lim_{n\rightarrow \infty} (1-\psi_n)^k - 1 = 0$. 
\end{proof}

\subsection{Proof of Proposition \ref{prop:st-phil}} \label{sec:MacroUpperBound}
In what follows, we denote the probability space by $(\Omega,\mathcal{A},P)$ and let   $\langle \mathcal{E}_t \rangle_{t \geq 1}$ be the filtration $\mathcal{E}_t = \sigma\{\sigma(X_1, \ldots, X_t) \cup\mathcal{N}\}$ where $\mathcal{N}$ is the family of sets of measure zero. The random times are $\tau_i:\Omega \rightarrow {\mathbb{N}}_+$ and they are $\mathcal{A}-\mathcal{P}({\mathbb{N}}_+)$ measurable, where $\mathcal{P}(\mathbb{N}_+)$ denotes the power set of $\mathbb{N}_+ = \{1,2\ldots\}$. Furthermore, let the random variables $X_i$ attain values in $\mathcal{X} \subset\mathbb{R}$ and assume that they are $\mathcal{A}-\mathcal{B}_\mathcal{X}$ measurable, where $\mathcal{B}_\mathcal{X}$ denotes the Borel $\sigma$-algebra on $\mathcal{X}$. We define the random times $\tau_i$ inductively together with a filtration that tracks the observed information. 
Let $\mathcal{H}_0 = \{\emptyset,\Omega\}$, let $\tau_1$ be a $\mathcal{H}_0$-measurable random time and let $\mathcal{H}_1 = \sigma(X_{\tau_1})$. Then, for given $\tau_1,\ldots,\tau_i$ and $\mathcal{H}_1,\ldots,\mathcal{H}_i$ let $\tau_{i+1}$ be some $\mathcal{H}_{i}$-measurable random variable such that $\tau_{i+1} > \tau_i$ almost surely and define $\mathcal{H}_{i+1} = \sigma(X_{\tau_1}, \ldots, X_{\tau_{i+1}})$. We can observe that $\sigma(X_{\tau_1}, \ldots, X_{\tau_i}) = \sigma(X_{\tau_1}, \ldots, X_{\tau_t}, \tau_1, \ldots, \tau_t)$. To see this, note that $\tau_1$ is $\sigma(X_{\tau_1},\ldots, X_{\tau_t})$-measurable, hence  $\sigma(X_{\tau_1}, \ldots, X_{\tau_t}, \tau_1, \ldots, \tau_t) = \sigma(X_{\tau_1}, \ldots, X_{\tau_t}, \tau_2, \ldots, \tau_t)$. Moreover, since $\tau_2$ is $\sigma(X_{\tau_1},\tau_1)= \sigma(X_{\tau_1})$-measurable, we have that 
$\sigma(X_{\tau_1}, \ldots, X_{\tau_t}, \tau_2, \ldots, \tau_t)~= \sigma(X_{\tau_1}, \ldots, X_{\tau_t}, \tau_3, \ldots, \tau_t)$, and the observation can simply be verified by induction.  Recall the statement of Proposition \ref{prop:st-phil} as follows. 

Assume that $\langle X_t \rangle_{t \in \mathbb{N}_+}$ is a stationary $\varphi$-mixing process with mixing coefficients $\langle \varphi_i \rangle_{i \in \mathbb{N}_+}$ such that $E X_t=\mu,~t \in \mathbb{N}_+$, and $\sup_{t \in \mathbb{N}_+ }|X_t| \leq c$ for some $c \in [0,\infty)$. Furthermore, let
 $\tau_1, \tau_2, \ldots$ be a sequence of random times such that $\tau_i + \ell \leq \tau_{i+1}$ a.s. for some $\ell \geq 1$ and all $i \in \mathbb{N}_+$, and all $\tau_{i+1}$ are $\sigma(X_{\tau_1}, \ldots, X_{\tau_i})$- measurable with $\tau_1 \in \mathbb{N}_+$ being a fixed time. Then for any $n \in \mathbb{N}_+$ 
\[
 \Bigl|\frac{1}{n} \sum_{i=1}^n E X_{\tau_i} - \mu \Bigr| \leq 2 c  \varphi_\ell.
\]      
\begin{proof}
The following technical result is important in this context. 
\begin{Alemma} \label{lem:simpleProcInFt}
 For all $i,t\geq 1$ and $A \in \mathcal{H}_{i-1}$ we have that $A \cap \{\tau_i = t\} \in \mathcal{E}_t$. In particular $\{\tau_i = t\} \in \mathcal{E}_t$.
\end{Alemma} 

\begin{proof}
 For 
 $i=1$ the result holds trivially since $A \cap \{\tau_i = t\}$ is either $\Omega$ or $\emptyset$ and since $\mathcal{E}_t$ is a $\sigma$-algebra it contains $\Omega$ and $\emptyset$. For any other $i \geq 2$ observe that if the claim holds for all $1 \leq j \leq i-1$ then since  $\tau_i$ is $\mathcal{H}_{i-1}$-measurable 
 and because $\tau_i > \tau_{i-1}$ almost surely implies that  $U := \{\tau_i = t\} \cap \{\tau_{i-1} \geq t \} \in \mathcal{N} \subset \mathcal{E}_t$ and  we have
 \[
 \{\tau_i = t \} = U \cup \bigcup_{s = 1}^{t-1} \{\tau_i =t\} \cap \{\tau_{i-1} = s\} \in \mathcal{E}_t.
 \]
We can also observe that $\{B \cap \{\tau_i = t\}: B \in \sigma(X_{\tau_i}) \} \subset \mathcal{E}_t$ because 
\[
\{B \cap \{\tau_i = t\}: B \in \sigma(X_{\tau_i}) \} = 
\{X_t^{-1}[B'] \cap \{\tau_i = t\} : B' \in \mathcal{B}_\mathcal{X}\} 
\]
and both   $X_t^{-1}[B']$ and $\{\tau_i = t \}$ lie in  $\mathcal{E}_t$. 

Also, for any $j \in \{1, \ldots, i-1\}$ we have $\{B \cap \{\tau_{i} = t \} : B \in \sigma(X_{\tau_j}) \} \subset \mathcal{E}_t $: let $U = \{\tau_i =t \} \cap \{\tau_j \geq t+1-(i-j) \}$ then $U \in \mathcal{N}$ and
\[
\{B \cap \{\tau_{i} = t \} : B \in \sigma(X_{\tau_j}) \} = U \cup \bigcup_{s=1}^{t-(i-j)} \{X_s^{-1}[B'] \cap \{\tau_i = t\} \cap \{\tau_j = s \} : B' \in \mathcal{B}_\mathcal{X}\}
\]
and $X_s^{-1}[B'] \cap \{ \tau_j = s\} \in \mathcal{E}_s \subset \mathcal{E}_t$. We have shown that 
\[
\Bigl\{\{\tau_i=t\} \cap \bigcap_{j=1}^{i-1} B_j  : B_j \in \sigma(X_{\tau_j}) \Bigr\} \subset \mathcal{E}_t.
\]
This implies directly that 
\[
\sigma\Bigl\{\{\tau_i=t\} \cap \bigcap_{j=1}^{i-1} B_j  : B_j \in \sigma(X_{\tau_j}) \Bigr\} \subset \mathcal{E}_t.
\]
By induction one can verify that  $\{A \cap \{\tau_i = t\} : A \in \sigma(X_{\tau_1},\ldots,X_{\tau_{i-1}})\}$ is included in the left side. %
\end{proof}
We also need the following observation: Let $s,t \in \mathbb{N}_+$, $s < t$ then for any $B \in \mathcal{E}_s$ we have $|\int_B (E(X_t|\mathcal{E}_s) - E(X_t))| \leq 2 \varphi_{t-s} c P(B)$.  This follows from Lemma \ref{lem:addZeroMeas} by remembering that $\mathcal{E}_s = \sigma( \sigma(X_1,\ldots,X_s) \cup \mathcal{N})$, i.e. the $\varphi$-mixing property implies this upper bound for $\sigma(X_1,\ldots,X_s)$ instead of $\mathcal{E}_s$ and Lemma \ref{lem:addZeroMeas} allows us to extend this property to $\sigma( \sigma(X_1,\ldots,X_s) \cup \mathcal{N})$ by using in Lemma \ref{lem:addZeroMeas} the following $\sigma$-algebras: $\mathcal{B} = \sigma(X_t), \mathcal{C} = \sigma(X_1,\ldots, X_s)$ and  $\mathcal{D} = \mathcal{N}$.

Now, coming to the main proof consider first $\tau_1$.  Note that $\tau_1$ is independent of any $X_i$ since for $U \in \sigma(X_i), V \in \sigma(\tau_1) = \{\emptyset,\Omega\}$ we have either $V = \Omega$ and $
P( U \cap V) = P(U) = P(U) P(V)$ or  $V = \emptyset$  and  $P( U \cap V) = 0 = P(U) P(V)$. Independence and stationarity of the process $\langle X_t \rangle_{t \in \mathbb{N}_+}$ give us for any $t \in \mathbb{N}_+$ that 
\[
E \left (X_t \times \chi{\{\tau_1 = t\}}\right ) = P(\tau_1 =t) E(X_t) = P(\tau_1 = t) E(X_1).
\]
Now, since the $X_t$ are bounded and $\tau_1$ attains a value in $\mathbb{N}_+$ we know that 
\[\sum_{t=1}^\infty E \left ( \abs{X_t} \times \chi{\{\tau_1 = t\}}\right )  = \sum_{t=1}^\infty P(\tau_1= t) E \abs{X_1}\] is finite and B. Levi's Theorem \citep{FREM}[123A] tells us that $\sum_{t=1}^\infty |X_t| \times \chi{\{\tau_1 = t\}}$ is integrable. Also $|\sum_{t=1}^{t'} X_t \times \chi{\{\tau_1 = t\}} |$ is upper bounded by the integrable function  $\sum_{t=1}^\infty |X_t| \times \chi{\{\tau_1 = t\}}$ for all $t'\in \mathbb{N}_+$ and Lebesgue's Dominated Convergence  Theorem \citep{FREM}[123C] gives us 
\begin{align*}
E(X_{\tau_1}) = E\bigl( \sum_{t=1}^\infty X_t \times \chi{\{\tau_1 = t\}} \bigr) = \sum_{t=1}^\infty E( X_t \times \chi{\{\tau_1 = t\}}) = \sum_{t=1}^\infty P(\tau_1 = t) E (X_t) = E( X_1). 
\end{align*}
We perform induction over $\tau_i$, $i \geq 2$. The 
set $\{\tau_i = t \}$ is an element of  $\mathcal{H}_{i-1}$ and for any $s < t$, $s,t \in \mathbb{N}_+$ we have that $B = \{ \tau_i = t\} \cap \{ \tau_{i-1} = s\} \in \mathcal{H}_{i-1}$. Observe that Lemma \ref{lem:simpleProcInFt} tells us that $B \in \mathcal{E}_s$ and, hence, that
\[
\int_B E(X_{\tau_i} | \mathcal{H}_{i-1})= \int_B X_{t}  = \int_B E(X_t | \mathcal{E}_s). 
\]
Due to stationarity this implies that
\begin{equation}
\bigl|\int_B (E(X_{\tau_i} | \mathcal{H}_{i-1}) - E(X_t))\bigr| = \bigl| \int_B (E(X_t | \mathcal{E}_s) - E(X_1)) \bigr| \leq 2 c \varphi_{t-s} P(B).  \label{eq:phi_l}
\end{equation}
Because $\tau_i$ is $\mathcal{H}_{i-1}$ measurable we can now write 
\begin{align}
E(X_{\tau_i}) 
&= E E(X_{\tau_i} | \mathcal{H}_{i-1})  \notag \\
& = E E(\sum_{t=1}^\infty X_t \times \chi{\{\tau_i = t \}} |\mathcal{H}_{i-1})\label{eq:inf_sum} \\
&= \sum_{t=1}^\infty E E(X_t \times \chi{\{\tau_i = t \}} |\mathcal{H}_{i-1}) \label{eq:move_sum}\\
&=
\sum_{t=1}^\infty E (E(X_t |\mathcal{H}_{i-1}) \times \chi{\{\tau_i = t \}})  \label{eq:chi_meaus}\\
&= \sum_{t=1}^\infty E( E(X_t |\mathcal{H}_{i-1}) \times \chi \Bigl(\{\tau_i = t \} \cap \bigcup_{s=1}^{t-\ell} \{\tau_{i-1} = s\}\Bigr)) \notag \\
&= \sum_{t=1}^\infty \sum_{s=1}^{t-\ell} E( E(X_t |\mathcal{H}_{i-1})  \times \chi \Bigl(\{\tau_i = t \} \cap \{\tau_{i-1} = s\}\Bigl), \label{eq1:Xtau}
\end{align}
where, with the same argument as above, using B. Levi's Theorem and Lebesgue's Dominated Convergence Theorem for the expectation operator and for the conditional expectation operator, the infinite summation in \eqref{eq:inf_sum} is moved outside to give \eqref{eq:move_sum}. Moreover,  \eqref{eq:chi_meaus} is due to  the fact that $\tau_i$ is $\mathcal H_{i-1}$-measurable. 
Finally, %substituting \eqref{eq1:Xtau} into $| E(X_{\tau_i}) - E (X_1) |$  and using 
%\eqref{eq:phi_l} with the event $B = \{\tau_i = t \} \cap \{\tau_{i-1} = s\} \in \mathcal{H}_{i-1}$ gives the following upper bound 
we obtain,
\begin{align}
&\left | E(X_{\tau_i}) - E (X_1) \right | \\
 &=\left | 
\sum_{t=1}^\infty \sum_{s=1}^{t-\ell} ( E( E(X_t |\mathcal{H}_{i-1})  \times \chi \left (\{\tau_i = t \} \cap \{\tau_{i-1} = s\}\right)) - P(\tau_i=t, \tau_{i-1} = s) E(X_1))\right | \label{eq:Xtausubs}\\
&\leq \sum_{t=1}^\infty \sum_{s=1}^{t-\ell} \bigl |( E( E(X_t |\mathcal{H}_{i-1})  \times \chi \left (\{\tau_i = t \} \cap \{\tau_{i-1} = s\}\right)) - P(\tau_i=t, \tau_{i-1} = s) E(X_1))\bigr| \notag\\
&\leq  2 c \sum_{t=1}^\infty \sum_{s=1}^{t-\ell}  \varphi_{t-s} P(\tau_i = t, \tau_{i-1}=s) \label{eq:phi_L_subs}\\ 
&\leq 2c \varphi_{\ell }\sum_{t=1}^\infty \sum_{s=1}^{t-\ell}  P(\tau_i = t, \tau_{i-1}=s) \label{eq:phi_dec} \\ 
&= 2 c \varphi_{\ell}, \notag
\end{align}
where, \eqref{eq:Xtausubs} follows from  \eqref{eq1:Xtau}, \eqref{eq:phi_L_subs} follows from \eqref{eq:phi_l} with $B := \{\tau_i = t \} \cap \{\tau_{i-1} = s\} \in \mathcal{H}_{i-1}$, and \eqref{eq:phi_dec} is due to the fact that $\langle \varphi_{n} \rangle_{n \in \N_+}$ is a decreasing sequence, see, e.g. \citep{bradley2007introduction}[vol. 1 pp. 69], and that $t-s \geq \ell$. This proves the statement. 
 \end{proof}

\subsection{Details of Example \ref{ex:nonMixingPolicy}} \label{app:ExamplenonMixingPolicy}
We give in this section the details of the construction in Example \ref{ex:nonMixingPolicy}. This example demonstrates that the sequence $\langle X_{\tau_n} \rangle_{n \in \mathbb{N}_+}$ of samples of one of the arms of a two armed bandit problem does not need to be $\varphi$-mixing even though the original process is $\varphi$-mixing. The argument uses standard results from Markov chains as they can be found in \citep{LEV08}  and a well known perturbation result for Markov chains. We provide the details of the argument for completeness.  

Assume we have a two arm bandit problem where the pay-off for arm two is zero at all time and the pay-off distribution of arm one is described by a Markov chain with two states, transition probabilities 
 $p_{11}= p_{22} = 1-\epsilon, p_{12}= p_{21} = \epsilon$, for some $\epsilon \in (0,1)$, and  probability $1/2$ to be in state $1$ at time $t=0$. The player gains a pay-off of $1$ if the Markov chain is in state $1$ and a pay-off of $0$ if the Markov chain is in state $2$. The Markov chain is irreducible  and aperiodic.  Furthermore, the Markov chain induces a stationary pay-off distribution. This implies that we are dealing with a jointly stationary $\varphi$-mixing process
 with mixing coefficients  $\varphi_k$ being upper bounded by  $\varphi_k \leq  (1-2 \epsilon)^k$ : one can derive the particular bound by considering an eigendecomposition of the transition matrix $T$ which yields eigenvalues $\lambda_1 = 1, \lambda_2 = 1-2 \epsilon$ and eigenvectors $u_1 = \sqrt{2} (1/2 \quad 1/2)^\top, u_2 = \sqrt{2} (1/2 \enspace -1/2)^\top$, i.e. with $U= (u_1 \quad u_2)$ and $\Lambda$ being the diagonal matrix with entries $\lambda_1$ and $\lambda_2$ we have $T = U \Lambda U^\top$. The stationary distribution over the states is $s= (1/2 \quad 1/2)^\top$ and for any vector $v = (v_1 \quad v_2)^\top, v_1,v_2 \geq 0, v_1 + v_2 = 1$, we have $\Lambda^k U^\top v = 
 (1/\sqrt{2}) (1 \quad (1-2\epsilon)^k (v_1-v_2))^\top$. Hence, $T^k v  - s = (1/2) (1-2\epsilon)^k(v_1 - v_2) (1 \quad -1)^\top$ and 
$\norm{T^k v - s}_\infty \leq (1/2) (1-2\epsilon)^k$. The mixing coefficients can now be bounded in the following way. Let $X_{t,i}, t \geq 1, i \in \{1,2\}$, be random variables that represent the pay-off of arm $i$ gained at time $t$. Consider a particular realization where $X_{1,1} = x_1, \ldots, X_{n,1} = x_n$ and $X_{n+k,1} = x_{n+k}, \ldots, X_{n + k + m,1} = x_{n+k+m}$ for some $x_1,\ldots,x_n, x_{n+k+1},\ldots, x_{n+k+m} \in \{0,1\}$ then $P(X_{1,1} = x_1, \ldots, X_{n,1} = x_n, X_{n+k,1} = x_{n+k}, \ldots, X_{n + k + m,1} = x_{n+k+m}) = P(X_{1,1} = x_1, \ldots, X_{n,1} = x_n) P(X_{n+k,1} = x_{n+k}, \ldots, X_{n + k + m,1} = x_{n+m+k}  | X_{n,1} = x_n)$ and 
 \begin{align}
&|P(X_{n+k,1} = x_{n+k}, \ldots, X_{n + k + m,1} = x_{n+m+k}  | X_{n,1} = x_n) \notag \\
& \hspace{3cm}- P(X_{n+k,1} = x_{n+k}, \ldots, X_{n + k + m,1} = x_{n+m+k})| \notag \\
&= P(X_{n+k+1,1} = x_{n+k+1}, \ldots, X_{n + k + m,1} = x_{n+m+k} | X_{n+k,1} = x_{n+k})  \notag
\\
&\hspace{3cm} \times | P(X_{n+k,1} = x_{n+k}) - P(X_{n+k,1} = x_{n+k}| X_{n,1} = x_n) | \notag \\
&\leq (1/2) (1-2 \epsilon)^k  P(X_{n+k+1,1} = x_{n+k+1}, \ldots, X_{n + k + m,1} = x_{n+m+k} | X_{n+k,1} = x_{n+k})  \notag \\
&= (1-2\epsilon)^k P(X_{n+k,1} = x_{n+k}, \ldots, X_{n + k + m,1} = x_{n+m+k}). \label{eq:markovchain_bnd}
\end{align}
since $P(X_{n+k,1} = x_{n+k}) = 1/2$.
Let $A = \{1,\ldots, n\}$, $B=\{n+k,\ldots,n+k+m\}$ and consider $\sigma(X_A), \sigma(X_B)$.  The events in these $\sigma$-algebras are finite unions of events of the form $\{X_{1,1} = x_1,\ldots, X_{n,1} = x_n\}$ and  $\{X_{n+k,1} = x_{n_k},\ldots, X_{n+k+m,1} = x_{n+k+m}\}$.

For any $U \in  \sigma(X_A)$ we know that $U$ consists at most of finitely many such events $U_1,\ldots, U_l$, $U_i \cap U_j = \emptyset$, for all $i,j \leq l$.  Similarly for $V \in \sigma(X_B)$ we know that $V = V_1 \cup \ldots \cup V_o$, $V_i \cap V_j = \emptyset$, for all $i,j \leq o$. The above argument which leads to the bound  \eqref{eq:markovchain_bnd} allows us to conclude that for any $U \in \sigma(X_A)$ and $V \in \sigma(X_B)$
\begin{align*}
|P(U)P(V) - P(U \cap V)| &= \sum_{i\leq l} \sum_{j \leq o} |P(U_i)P(V_j) - P(U_i \cap V_j)|  \\
&\leq   (1-2\epsilon)^k \sum_{i\leq l} \sum_{j \leq o} P(U_i) P(V_j) 
\leq (1-2\epsilon)^k.
\end{align*}%a
%The argument works in the same way for sets $A$ and $B$ that do not consist of consecutive indices.
Hence the Markov chain is $\varphi$-mixing. The second arm has a constant reward and does not introduce any dependencies in the pay-off over time. Therefore we have a jointly stationary $\varphi$-mixing process with the mixing coefficient being equal to the mixing coefficients of the Markov chain.

% and since the second arm is producing a constant reward of $0$ we also know that the bandit problem is $\varphi$-mixing with the same mixing coefficient. 

Now fix some $\delta>0$ and consider the following policy $\pi^\delta$. At $t=1$ the policy plays arm $1$ receiving pay-off $X_{1,1}$. Then at any other $t\geq 2$ the arm is selected according to the  following rules: if at $t-1$ arm $1$ has been played and $X_{t-1,1} = X_{1,1}$ then the policy chooses at $t$ arm $1$; if at $t-1$ arm $1$ has been played and $X_{t-1,1} \not = X_{1,1}$ then the policy chooses at $t$ arm $2$ and plays arm $2$ for the next  $k := \lceil \log(2 \delta)/\log(1-2\epsilon) \rceil$ rounds before switching back to arm $1$. We'd like to show that the sequence of pay-offs generated by policy $\pi^\delta$ at arm $1$ is not $\varphi$-mixing. Let $\tau_1,\tau_2,\ldots $ be the sequence of random times at which arm $1$ is played (by construction $\tau_1 = 1$ and $\tau_2 = 2$). The evolution of the process $ \langle X_{\tau_n} \rangle_{n \geq 1}$ can also be described by transition matrices.  The transition probabilities to move from a state at $t=1$ to a state at $t=2$ are just the probabilities summarized in $T$. For all other $t$  ($t\geq 2$) the transition matrix is either 
\begin{equation} \label{eq:tildeTdef}
\tilde T = \begin{pmatrix}
1- \epsilon & \epsilon \\
(T^k)_{21} & (T^k)_{22}
\end{pmatrix} \text{\quad or \quad} \tilde T = \begin{pmatrix}
(T^k)_{11} & (T^k)_{12} \\
\epsilon & 1- \epsilon 
\end{pmatrix}
\end{equation}
depending on $X_{\tau_1}$, i.e. if $X_{\tau_1} = 1$ then the former is describing the evolution and if $X_{\tau_1} =0$ the latter is the one describing the evolution of the Markov chain. In the following we discuss the case that $X_{\tau_1} = 1$, but the same arguments apply to the case $X_{\tau_1} = 0$. We can observe  that  $\|v^\top (\tilde T - \hat T)\|_\infty  \leq \delta$ for all $v$ with non-negative entries and $v_1+ v_2 = 1$, where 
\[
\hat T  =  \begin{pmatrix}
1- \epsilon & \epsilon \\
1/2 & 1/2
\end{pmatrix}
\]
in case that $X_{\tau_1} = 1$. The claim can be verified through 
\begin{align*}
\norm{v^\top \tilde T  - v^\top \hat T}_\infty 
&= \norm{ v^\top\begin{pmatrix}
0 & 0 \\ 
\tilde T_{21} - 1/2 & \tilde T_{22} - 1/2 
\end{pmatrix}  }_\infty  \\
&= v_2 \norm{T^k \begin{pmatrix}
0 \\ 1
\end{pmatrix} - \begin{pmatrix} 1/2 \\  1/2 \end{pmatrix}}_\infty \leq v_2 \delta \leq \delta.
\end{align*}
The Markov chains associated to $\tilde T$ and $\hat T$ are both irreducible and aperiodic. This implies in particular the existence of stationary distributions, with the associated probabilities to be in state one and two summarized in vectors $\tilde s,\hat s \in [0,1]^2$, and the convergence of $\tilde T^l,\hat T^l$ to $\tilde s, \hat s$ in $l$ (measured in $\norm{}_\infty$, \citep{LEV08}[Thm. 4.9]).  In particular, there exist constants $\tilde c,\hat c >0$ and $\tilde \alpha, \hat \alpha \in (0,1)$ such that for all $l\geq 1$ and any vector $v \in [0,1]^2$ with $v_1 + v_2 = 1$ 
\[
\| v^\top \tilde T^l - \tilde s \|_\infty \leq \tilde c \tilde \alpha^l \text{\quad and \quad} \| v^\top \hat T^l  - \hat s \|_\infty \leq \hat c \hat \alpha^l.
\] 
We can also calculate the stationary distribution of $\hat T$ explicitly to get $\hat s = (1/(1+2\epsilon) \enspace 2\epsilon/(1+2\epsilon))^\top$.

It is known that Markov chains with slightly perturbed transition matrices have similar stationary distributions. Due to \cite{CHO01} there exists a constant $c > 0$ that is only dependent on $\hat T$ (and independent of $\tilde T$) such that 
$\|\tilde s - \hat s \|_\infty \leq c \| \tilde T - \hat T\|_{\infty,1} \leq 2 \delta c $
where $ \| \tilde T - \hat T\|_{\infty,1} := \max_{i \in \{1,2\}} \sum_{j=1}^2 |\tilde T_{i,j} - \hat T_{i,j}| \leq 2 \delta$. Combining these inequalities yields  
$\| v^\top \tilde T^l - \hat s\|_\infty \leq 2 \delta c + \tilde c \tilde \alpha^l$ for any $v$ with non-negative entries and $v_1 + v_2 = 1$.

Consider now the events $U= \{X_{\tau_1,1} = 1\}$ and $V = \{X_{\tau_n,1} = 1\}$ for $n \geq 2$. We have $P(U) = 1/2 = P(V)$ where the second equality follows from
\begin{align*}
2 P(X_{\tau_n,1} = 0) &=  P(  X_{\tau_n,1} = 0 | X_{\tau_1,1} = 0 ) + P(  X_{\tau_n,1} = 0 | X_{\tau_1,1} = 1) \\ &=  P(  X_{\tau_n,1} = 1 | X_{\tau_1,1} = 1 ) + P(  X_{\tau_n,1} = 1 | X_{\tau_1,1} = 0) \\
&= 2 P(X_{\tau_n,1} = 1).
\end{align*}
 Furthermore,
with $\tilde T$ being the matrix defined on the left side of (\ref{eq:tildeTdef})
\begin{align*}
\Bigl|\frac{P(U \cap V)}{P(U)} - \frac{1}{1+2\epsilon}\Bigr| %
&=    \Bigl| (1- \epsilon \enspace \epsilon) \tilde T^{n-1} \begin{pmatrix} 1 \\  0  \end{pmatrix} -  \frac{1}{1+2\epsilon} \Bigr| \\
&\leq  \Bigl\| (1- \epsilon \enspace \epsilon) \tilde T^{n-1} - \frac{1}{1 + 2 \epsilon} \begin{pmatrix} 1 \\  2\epsilon  \end{pmatrix}  \Bigr\|_\infty
\end{align*}
and the last term is upper bounded by $2 \delta c + \tilde c \tilde \alpha^{n-1}$. Recalling that $c$ does not depend on $\tilde T$ and, hence, not on $\delta$ we see that we can make the term $2 \delta c$ arbitrary small. Furthermore, by considering a large $n$ we can make the second term arbitrary small. In particular, let $\epsilon = 1/10$ then there exists a $\pi^\delta$ and an $N \in \mathbb{N}$ such that for all $n\geq N$ 
\begin{equation} \label{eq:stat_of_cond}
\Bigl|\frac{P(U \cap V)}{P(U)} - \frac{1}{1+2\epsilon}\Bigr| \leq 1/10.
\end{equation}
Hence, for all $n\geq N$ 
\begin{align*}
|P(U)P(V) -  P(U \cap V)| &\geq P(U) \Bigl(  
\Bigl| \frac{1}{2} -  \frac{1}{1+2\epsilon}\Bigr| - 1/10 \Bigr) \\ 
&\geq  P(U)/5
\end{align*}
and the process is not $\varphi$-mixing.
\section{Proof of Theorem~\ref{thm:regret_azal}: Regret Bound for $\varphi$-mixing Bandits}
\label{app:thm:regret_azal} 
For the regret $\overline{\mathcal R}(n)$ of Algorithm~\ref{alg:ucb} after $n$ rounds of play. We have,
\begin{align*}
 \overline{\mathcal R}(n)&\leq \sum_{\substack{i=1\\ \mu_i \neq \mu^*}}^k\frac{32(1+8\norm{\varphi})\ln n}{\Delta_i}+(1+2\pi^2/3)(\sum_{i=1}^k\Delta_i)+\norm{\varphi}\log n
\end{align*}
\begin{proof}
Thanks to Proposition~\ref{lemma22}, in order to bound the regret $\overline{\mathcal R}(n)$ it suffices to calculate the expected number of times $T_i(n)$ that a suboptimal arm is played in $n$ rounds. 
For any $s, t \in \N_+$, let $c_{t,s}:=\sqrt{(8 \zeta ((1/8)+\ln t))/2^s}+\norm{\varphi}/2^{s-1}$, where $\zeta=1+8\norm{\varphi}$. Recall that Algorithm~\ref{alg:ucb} plays its selected arms in batches of exponentially growing length so that if arm $j$ for $j \in 1..k$ is selected at round $t$, it is played for 
$2^{s_j(t)}$ consecutive time-steps, where $s_j(t)$ denotes the number of times that arm $j$ has been selected up-to time $t$. 
Below, we denote by $\overline{X}_{j,s}:=\frac{1}{2^{s}}\sum_{t'=t}^{t+2^{s}-1}X_{t',j}$ with $s:=s_j(t)$, the Algorithm's estimate of the stationary mean of arm $j$ selected at time $t$.
As usual, a superscript ``$*$'' refers to the quantities for the arm with the highest stationary mean. 
Fix some $l \in \N_+$. We have,
\begin{align*}
T_i(n)  &= 1+\sum_{t=k+1}^n \chi\{\pi_t=i\} \\
        & \leq 2^{l+1}+ \sum_{t=2^{l+1}+1}^n \sum_{m=l+1}^{\log t} 2^m \chi\{\tau_{m,i}=t\} \\
        & \leq 2^{l+1}+ \sum_{t=2^{l+1}+1}^n \sum_{m=l+1}^{\log t} 2^m \chi\Bigl\{\min_{s=1..\log t}\overline{X}^*_{s}+c_{t,s}\leq \frac{1}{2^{m-1}}\sum_{u=t}^{t+2^{m-1}-1}X_{u,i}+c_{t,m-1}\Bigr\}\\
        & \leq 2^{l+1}+ \sum_{t=1}^{\infty} \sum_{m=l+1}^{\log t}\sum_{s=1}^{\log t}2^m \chi\Bigl\{\overline{X}^*_{s}+c_{t,s}\leq \frac{1}{2^{m-1}}\sum_{u=t}^{t+2^{m-1}-1}X_{u,i}+c_{t,m-1}\Bigr\}
\end{align*}
For every $t \in \N$ and every $s \in 1..\log t$ we have that $\overline{X}^*_{s}+c_{t,s}\leq \frac{1}{2^{m-1}}\sum_{u=t}^{t+2^{m-1}-1}X_{u,i}+c_{t,m-1}$ implies that
\begin{align}
\overline{X}^*_{s}&\leq \mu^*-c_{t,s}\\
\frac{1}{2^{m-1}}\sum_{u=t}^{t+2^{m-1}-1}X_{u,i}&\geq \mu_i+c_{t,m-1}\\
\mu^*&<\mu_i+2c_{t,m-1}\label{eq:delta}
\end{align}
Now, observe that for a fixed $t \in 1..n$ we have, 
\begin{align}
P(\overline{X}^*_{s}\leq \mu^*-c_{t,s})
&\leq P\left (\left |\overline{X}^*_{s}-\mu^*\right |\geq c_{t,s}\right )\nonumber \\
&\leq P \left (\left |\sum_{j=0}^{2^s-1} X^*_{\tau_{s}+j}-E X^*_{\tau_{s}+j} \right |+\left |\sum_{j=0}^{2^s-1}(E X^*_{\tau_{s}+j}-\mu^*)\right |\geq 2^s c_{t,s} \right )\nonumber \\
&\leq P\left (\left |\sum_{j=0}^{2^s-1} X^*_{\tau_{s}+j}-E X^*_{\tau_{s}+j} \right |\geq 2^s c_{t,s}-2\norm{\varphi}\right)\label{azal:eq:prop1}\\
&\leq \sqrt{e}\exp\Bigl(-\frac{(2^sc_{t,s}-2\norm{\varphi})^2}{2^{s+1}\zeta}\Bigr) \label{azal:hoeff}\\
&\leq t^{-4}, \label{azal:eq:t-4:1}
\end{align}
where, \eqref{azal:eq:prop1} follows from Lemma~\ref{lemma11} and \eqref{azal:hoeff} follows from a Hoeffding-type bound for $\varphi$-mixing processes given by Corollary 2.1 of \cite{RIO99} which is applicable due to Lemma \ref{lem:stopped_phi} (pp. \pageref{lem:stopped_phi}).
Moreover, noting that $\norm{\varphi}\geq 0$ we similarly obtain, 
\begin{align}
 P\Bigl(\frac{1}{2^{m-1}}\!\! \sum_{u=t}^{t+2^{m-1}-1} \!\! X_{u,i}\geq \mu_i+c_{t,m-1}\Bigr)&\leq P\Bigl(\frac{1}{2^{m-1}}\!\!\sum_{u=t}^{t+2^{m-1}-1}\!\!X_{u,i}\geq \mu_i+c_{t,m-1}-\frac{2}{2^{m-1}}\norm{\varphi}\Bigr)\nonumber \\
&\leq P\Bigl(|\!\!\sum_{u=t}^{t+2^{m-1}-1}\!\!X_{u,i}- 2^{m-1}\mu_i|\geq2^{m-1}c_{t,m-1}-2\norm{\varphi}\Bigr)\nonumber \\
&\leq t^{-4} \label{azal:eq:t-4:2}.
\end{align}
Let $L:=\log \frac{32(1+4\norm{\varphi})\ln n}{\Delta_i^2}$. Since, for $t \geq 2^L+1 $ and every $m \geq L+1$ we have $\mu^*-\mu_i-2c_{t,m-1} \geq 0$,
it follows that \eqref{eq:delta} is false for all $m \geq L+1$.
Therefore we have,
\begin{align}
 E(T_i(n))  
&\leq 2^{L+1}+\sum_{t=1}^{\infty}\sum_{m=L+1}^{\log t}\sum_{s=1}^{\log t}P(\overline{X}^*_{s}\leq \mu^*-c_{t,s} , \frac{1}{2^{m-1}}\sum_{u=t}^{t+2^{m-1}-1}X_{u,i}\geq \mu_i+c_{t,m-1})\nonumber \\
 &\leq 2^{L+1}+\sum_{t=1}^{\infty}\sum_{m=L+1}^{\log t}\sum_{s=1}^{\log t}2^{m+1}t^{-4}\label{azal:eq:exp3}\\
&\leq \frac{32(1+4\norm{\varphi})\ln n}{\Delta_i^2}+1+2\pi^2/3 \nonumber
\end{align}
where \eqref{azal:eq:exp3} follows from \eqref{azal:eq:t-4:1} and \eqref{azal:eq:t-4:2}.
\end{proof}
\section{Proofs for Strongly Dependent Reward Distributions}
 \subsection{Basic Bounds} \label{suppl:basic_bounds}
 Consider two independent and normally distributed random variables $X,Y$ with mean $\mu_X > \mu_Y$ and variance $\sigma_X^2,\sigma_Y^2$. Let $\Delta = \mu_X - \mu_Y >0$ and $\sigma^2 = \sigma_X^2 + \sigma_Y^2$. We derive in this section the following bounds which are crucial for the derivation of the regret. We use the notation $z^+ = \max\{0,z\}$ and $\phi(x)= (2 \pi)^{-1/2} \exp(-x^2/2)$ for the density function of the standard normal distribution.
\begin{align}
E(Y - X)^+ &\leq  \sigma \phi(\Delta/\sigma), \label{eq:gp1} \\
E(Y-X)^+ &\geq 0,  \label{eq:gp2} \\
E(X-Y)^+ &\leq  \sigma \phi(\Delta/\sigma) + \Delta, \label{eq:gp3}\\
E(X-Y)^+ &\geq \Delta. \label{eq:gp4}
\end{align}
The derivation is based on basic properties of Gaussian random variables. Recall that $Z = X-Y$ is normally distributed with mean $\Delta$ and variance $\sigma^2$. Therefore,
 \begin{align*}
\sqrt{2\pi} \sigma  E(X-Y)^+ &=  \int_{0}^\infty z \exp\left(- \frac{(z-\Delta)^2}{2 \sigma^2}\right) \\ 
 &=\int_{-\Delta}^\infty z \exp\left(- \frac{z^2}{2 \sigma^2}\right) 
 +  \Delta \int_{-\Delta}^\infty \exp\left(- \frac{z^2}{2 \sigma^2}\right)  \\
&=  
\sigma^2 \exp\left(- \frac{\Delta^2}{2 \sigma^2} \right)
 + \Delta \sigma \int_{-\infty}^{\Delta/\sigma} \exp\left(- \frac{z^2}{2}\right)
\end{align*}
%and inequality \eqref{eq:gp3} follows.
Standard bounds on the cdf, as can be found in \citep{DUDLEYProb}[Lem. 12.1.6], lead to \eqref{eq:gp4}, i.e. 
\begin{align*}
\sqrt{2\pi} \sigma  E(X-Y)^+ &= \sigma^2 \exp\left(- \frac{\Delta^2}{2 \sigma^2} \right) 
+ \Delta \sigma \sqrt{2\pi} \left(1- \frac{1}{\sqrt{2\pi}} \int_{\Delta/\sigma}^\infty \exp\left(- \frac{z^2}{2}\right)\right) \\
&\geq \sigma^2 \exp\left(- \frac{\Delta^2}{2 \sigma^2} \right) 
+ \Delta \sigma \sqrt{2\pi} \left(1- \frac{\sigma}{\sqrt{2\pi} \Delta} \exp\left(-\frac{\Delta^2}{2\sigma^2} \right)\right) \\
&= %(\sigma^2 - \Delta \sigma) \exp\left(- \frac{\Delta^2}{2 \sigma^2} \right) + 
\Delta \sigma \sqrt{2\pi}.
\end{align*}
Similarly, if we consider $Z = Y-X$ which has mean $-\Delta$
\begin{align*}
\sqrt{2\pi} \sigma  E(Y-X)^+ &=  \int_{0}^\infty z \exp\left(- \frac{(z+ \Delta)^2}{2 \sigma^2}\right) \\ 
% &=\int_{\Delta}^\infty z \exp\left(- \frac{z^2}{2 \sigma^2}\right) 
% -  \Delta \int_{\Delta}^\infty \exp\left(- \frac{z^2}{2 \sigma^2}%\right)  \\
% &=  \frac{1}{\sqrt{2 \pi} \sigma} \left( (1/2)\int_{0}^\infty  \exp%\left(- \frac{z}{2 \sigma^2}\right) 
 %- (1/2)\int_{0}^{\Delta^2}  \exp\left(- \frac{z}{2 \sigma^2}\right)
 %+  \Delta \sigma \int_{-\Delta/\sigma}^\infty \exp\left(- \frac{z^2}{2}\right) \right) \\
&=  
\sigma^2 \exp\left(- \frac{\Delta^2}{2 \sigma^2} \right)
 -  \Delta \sigma \int_{\Delta/\sigma}^{\infty} \exp\left(- \frac{z^2}{2}\right).
\end{align*}
Applying the result from \cite{DUDLEYProb}[Lem. 12.1.6] leads here to the trivial lower bound $0$ and, hence, inequality \eqref{eq:gp2}.
\iffalse
and we get the following lower bound
\begin{align*}
\sqrt{2\pi} \sigma  E(Y-X)^+ &= \sigma^2 \exp\left(- \frac{\Delta^2}{2 \sigma^2} \right) 
- \Delta \sigma \sqrt{2\pi} \frac{1}{\sqrt{2\pi}} \int_{\Delta/\sigma}^\infty \exp\left(- \frac{z^2}{2}\right) \\
%&\geq \sigma^2 \exp\left(- \frac{\Delta^2}{2 \sigma^2} \right) 
%+ \Delta \sigma \sqrt{2\pi} \left(1- \frac{1}{\sqrt{2\pi}} \exp\left(-%\frac{\Delta^2}{2\sigma^2} \right)\right) \\
&\geq (\sigma^2 - \sigma^2) \exp\left(- \frac{\Delta^2}{2 \sigma^2} \right) = 0.
\end{align*}
\fi
\iffalse
\begin{align*}
\sqrt{2\pi} \sigma  E(Y-X)^+ &= \sigma^2 \exp\left(- \frac{\Delta^2}{2 \sigma^2} \right) 
- \Delta \sigma \sqrt{2\pi} \left(1- \frac{1}{\sqrt{2\pi}} \int_{\Delta/\sigma}^\infty \exp\left(- \frac{z^2}{2}\right)\right) \\
%&\geq \sigma^2 \exp\left(- \frac{\Delta^2}{2 \sigma^2} \right) 
%+ \Delta \sigma \sqrt{2\pi} \left(1- \frac{1}{\sqrt{2\pi}} \exp\left(-%\frac{\Delta^2}{2\sigma^2} \right)\right) \\
&\leq (\sigma^2 + \Delta \sigma) \exp\left(- \frac{\Delta^2}{2 \sigma^2} \right) - \Delta \sigma \sqrt{2\pi}.
\end{align*}
\fi
We use the following inequalities to obtain upper bounds on $E(Y-X)^+$ and $E(X-Y)^+$. Let $Z$ be a standard normal random variable. Then 
\[
\Pr(Z \geq c) \geq  \begin{cases}
\phi(c)/(2c) &\text{if } c\geq 1, \\
 \phi(c) (1-c)/2 & \text{if } 0 \leq c < 1.
\end{cases}
\]
The first bound can be found in \citep{DUDLEYuclt}. The second bound is a straightforward adaptation of the techniques used to derive the first bound.
Applying these bounds we get the following upper bounds on $E(Y-X)^+$. If $\Delta/\sigma \geq 1$ then 
\begin{align*}
\sqrt{2\pi} \sigma  E(Y-X)^+
&\leq  \sigma^2 \exp\left(- \frac{\Delta^2}{2 \sigma^2} \right) - \Delta \sigma \frac{\sigma}{2 \Delta } \exp\left(-\frac{\Delta^2}{2\sigma^2} \right)  \\
&= \frac{\sigma^2}{2} \exp\left(- \frac{\Delta^2}{2 \sigma^2} \right) 
\end{align*}
and if $0 \leq \Delta/\sigma < 1$ 
\begin{align*}
\sqrt{2\pi} \sigma  E(Y-X)^+
&\leq  \sigma^2 \exp\left(- \frac{\Delta^2}{2 \sigma^2} \right) - \Delta \sigma \frac{(1-\Delta/\sigma)}{2  } \exp\left(-\frac{\Delta^2}{2\sigma^2} \right)   \\
&= (\sigma^2 - \Delta \sigma/2+  \Delta^2/2) \exp\left(- \frac{\Delta^2}{2 \sigma^2} \right) \\
&\leq  \sigma^2 \exp\left(- \frac{\Delta^2}{2 \sigma^2} \right)
% \\
%&\leq (\sigma - \Delta/\sqrt{2})^2 \exp\left(- \frac{\Delta^2}{2 %\sigma^2} \right).
\end{align*}
and the inequality \eqref{eq:gp3} follows. The upper bounds on $E(X-Y)^+$ are derived in the same way.
 For $\Delta/\sigma \geq 1$ these are
\begin{align*}
\sqrt{2\pi} \sigma E(X-Y)^+ %&=  
%\sigma^2 \exp\left(- \frac{\Delta^2}{2 \sigma^2} \right)
% + \Delta \sigma \sqrt{2\pi} \left(1-  \frac{1}{\sqrt{2\pi}}  %\int_{\Delta/\sigma}^{\infty} \exp\left(- \frac{z^2}{2}\right) \right) \\
 &\leq  \sigma^2 \exp\left(- \frac{\Delta^2}{2 \sigma^2} \right) - (\sigma^2/2) \exp\left(- \frac{\Delta^2}{2 \sigma^2} \right)  + \Delta \sigma \sqrt{2\pi} \\
 &= \frac{\sigma^2}{2} \exp\left(- \frac{\Delta^2}{2 \sigma^2} \right)
+\Delta \sigma \sqrt{2\pi}
\end{align*}
and for $0 \leq \Delta/\sigma \leq 1$
\begin{align*}
\sqrt{2\pi} \sigma E(X-Y)^+ %&=  
%\sigma^2 \exp\left(- \frac{\Delta^2}{2 \sigma^2} \right)
% + \Delta \sigma \sqrt{2\pi} \left(1-  \frac{1}{\sqrt{2\pi}}  %\int_{\Delta/\sigma}^{\infty} \exp\left(- \frac{z^2}{2}\right) \right) %\\
 &\leq  \sigma^2 \exp\left(- \frac{\Delta^2}{2 \sigma^2} \right) - \frac{\Delta \sigma(1-\Delta/\sigma)}{2} \exp\left(- \frac{\Delta^2}{2 \sigma^2} \right)  + \Delta \sigma \sqrt{2\pi} \\
 &= (\sigma^2 - \Delta \sigma/2 +\Delta^2/2)\exp\left(- \frac{\Delta^2}{2 \sigma^2} \right) 
+ \Delta \sigma \sqrt{2\pi}
\end{align*}
and the inequality \eqref{eq:gp1} follows.

\subsection{Proof of Proposition \ref{steffen:GP}} \label{sec:UpperBoundGP}
We bound the regret of the two phases individually. Some technical steps are moved further below to streamline the discussion.

\paragraph{Regret of Phase I.}
We sweep through all $k$ arms at times $1$ to $k$, $m +1$ to $m+k$, etc. During each of these phases we build up regret. This regret can be bounded by
\begin{align}
\sum_{i=1}^k E(\max_{j\not = i } X_{m+i,j} &- X_{m+i,i})^+  \notag \\
&\leq  \sum_{i=1}^k \sum_{j \not = i} E(X_{m+i,j} - X_{m+i,i})^+ \notag \\
&\leq  (k-1)\sum_{i=1}^k  \max_{j \not = i}  E(X_{m+i,j} - X_{m+i,i})^+ \notag \\
&\leq (k-1) \sum_{i=1}^{k} \Delta_{i} + \sqrt{2}(k-1) \sum_{i=1}^k \phi(\Delta_i/\sqrt{2}) \notag \\
&\leq (k-1) \sum_{i=1}^k (\Delta_{i} + \sqrt{2}), \label{eq:regretPhaseI}
\end{align}
where $\Delta_i = \mu^* - \mu_i$ is the difference between the highest stationary mean of the $k$ arms and the stationary mean of arm $i$. Here, we use Inequality \eqref{eq:gp3} and the assumption that the variance of the individual processes is $1$. Therefore, the $\sigma$ appearing in the bound \eqref{eq:gp3} is $\sqrt{2}$ since $\sigma^2$ is the sum of the variances of the two individual processes.

\iffalse
\[
\sum_{i=1}^{k} \max_{j\leq k} E(X_{\tau+i-1,j} - X_{\tau+i-1,i})^+ 
\leq \sum_{i=1}^{k} E(X_{\tau+i-1,i^*} - X_{\tau+i-1,s})^+.
\]
$i^*$ denotes here the arm with the highest stationary mean. Using the bounds from Equation \ref{eq:bndsGP},
\[
\sum_{i=1}^k E(X_{\tau+i-1,i^*} - X_{\tau+i-1,i})^+ \leq \sum_{i=1}^k \Delta_{i} + 2 \sum_{i=1}^k \phi(\Delta_i/2). 
\]
\fi

\paragraph{Regret of Phase II.} To control the regret building up in the second phase we condition on the observations in a sweep at time $l m$, $ l\in \mathbb{N}$, i.e. on the observed pay-offs $x_1, \ldots, x_k$.  Arm $i^*$ is selected such that $x_{i^*} \geq \max_{i\leq k} x_i$ and this arm is played for $m-k$ steps. 
We need to control $E(\max_{i \leq k} (X_{t',i} - X_{t',\pi_{t'}})^+)$ for all $lm +k +1 \leq t' \leq (l+1)m$. Due to stationarity this is equal to  controlling $E(\max_{i \leq k} (X_{t,i} - X_{t,\pi_t})^+)$, for all $k+1 \leq t \leq m$.  

We use in the following $P_{x_1,\ldots, x_k}$, $P_{x_i,x_u}$ for the conditional distributions given that $X_{11} =x_1, \ldots, X_{kk} = x_k$ and $X_{ii}=x_i, X_{uu}= x_u$, and we use $\nu$ for the marginal measure. In Appendix \ref{sec:conditioning} below we show that  for any $k+1 \leq t \leq m$ it holds that
\begin{align}
E(\max_{i \leq k} (X_{t,i} &- X_{t,\pi_t})^+)  \notag \\
&\leq \sum_{u=1}^k \sum_{i \not =  u} \int  P_{x_1,\ldots,x_k}( u = i^*)  \times \int  (X_{t,i} - X_{t,u})^+ dP_{x_i, x_u} d\nu(x_1,\ldots,x_k). \label{eq:marginal_integral}
\end{align}
%\[
%E\Bigl(\max_{i \leq k} (X_{t,i} - X_{t,\pi_t})^+\Bigr)  \leq \sum_{u=1}^k \sum_{i \not = u} \int  \chi\{u = i^*\} \int  (X_{t,i} - X_{t,u})^+ dP_{x_i,x_u} d\nu(x_i,x_u)
%\]
%where we used $P_{x_i,x_u}$ for the conditional distribution of $X_{t,i}$ and $X_{t,u}$ given the values $X_{i,i} =x_i$ and $X_{u,u} = x_u$ (see below in Appendix \ref{sec:conditioning} for details).   
\iffalse
Denote the $\sigma$-algebra $\sigma\{X_{1,1}, \ldots,    X_{k,k} \}$ with $\mathcal{C}$ and let $i^*$ be the random variable that represents the arm selection, then 
\begin{align*}
&E\Bigl(\max_{i \leq k} (X_{t,i} - X_{t,\pi_t})^+\Bigr) 
\leq \sum_{u=1}^k E\Bigl(  \sum_{i \not= u} E(\chi\{ u = i^*\} \times  (X_{t,i} - X_{t,i^*})^+ | \mathcal{C})\Bigr)
\end{align*}
event that $\{X_{lm+1,1} = x_1, \ldots,    X_{lm+k,k} = x_k\}$ by $B$. Observe that 
\fi
The inner integral in \eqref{eq:marginal_integral} can be bounded by using inequality \eqref{eq:gp1}
\begin{equation}
\int  (X_{t,i} - X_{t,u})^+ dP_{x_i,x_u} \leq  \sigma_{t,i} \phi\left(\frac{\Delta_{t,i}}{\sigma_{t,i}} \right) \!, \label{eq:cond_int_bound}
\end{equation}
where $\Delta_{t,i} = E(X_{t,u} - X_{t,i} | X_{i,i} = x_i, X_{u,u} = x_{u}) $ which equals $E(X_{t,u}  | X_{u,u} = x_{u}) - E(X_{t,i} | X_{i,i} = x_i)$ due to the independence between arm $u$ and $i$. These are posterior means of Gaussian processes given observations $x_u$ and $x_i$, respectively. Observe that the posterior mean for $X_{t,u}$ is related to the posterior mean of the zero mean Gaussian process $X_{t,u} -\mu_u$ through $E(X_{t,u}  | X_{u,u} = x_{u}) = \mu_u + E(X_{t,u} -\mu_u | X_{u,u} - \mu_u = x_{u}- \mu_u)$. The posterior equation for the zero mean Gaussian process $X_{t,u} -\mu_u$ given observation $x_u-\mu_u$ at time $u$ is $\cov(t-u)(x_u-\mu_u)$ since $\cov(0) = 1$. Therefore, $E(X_{t,u}  | X_{u,u} = x_{u}) = \mu_u + \cov(t-u)(x_u - \mu_u)$ and, hence,  $E(X_{t,u}  | X_{u,u} = x_{u}) = (1-\cov(t-u))\mu_u +\cov(t-u) x_u$. Similarly,   $E(X_{t,i}  | X_{i,i} = x_{i}) = (1-\cov(t-i))\mu_i +\cov(t-i) x_i$
By defining  $\tilde \Delta_{i} = x_{u} - x_i$ and by using the H\"older-continuity assumption, we obtain the following  lower bound. 
\begin{align}
\Delta_{t,i} &= (\mu_{u} - \mu_i)(1- \cov(t-k)) + \tilde \Delta_{i} \cov(t-k) +  \varepsilon(t) \notag \\
%&\geq - \Delta (1-\cov(s)) +  \tilde \Delta_{\tau,i} \cov(s) - %\abs{\varepsilon(\tau,s)} \notag \\
&\geq  \tilde \Delta_i -|\mu_{u} - \mu_i| |\cov(0)- \cov(t-k)| - \tilde \Delta_i |\cov(0) - \cov(t-k)|  - |\varepsilon(t)| \notag \\
&\geq \tilde \Delta_{i} - c(t-k)^\alpha (\Delta + \tilde \Delta_{i})  - 
\abs{\varepsilon(t)} \label{eq:delta_t_i_bound}
\end{align}
with $\varepsilon(t)$ being a term which can be bounded by  
$\abs{\varepsilon(t)} \leq (\Delta + \tilde \Delta_{i}) ck^\alpha$ (see Appendix \ref{sec:varepsilon}).
For the conditional variance, 
\begin{align*}\sigma^2_{t, i } = E(X^2_{t,u} | X_{u,u} 
 = x_{u}) - E(X_{t,u} | X_{u,u} 
 = x_{u})^2 + E(X^2_{t,i} | X_{i,i} 
 = x_{i}) - E(X_{t,i} | X_{i,i} 
 = x_{i})^2,
 \end{align*} 
 we obtain the following upper-bound   
\begin{equation}
\sigma_{t,i}^2 = 2  - \cov^2(t - u) - \cov^2(t- i)  \leq  2c(t - u)^\alpha + 2 c(t-i)^\alpha
 \leq 4 c t^\alpha, \label{eq:variance_bound}
\end{equation}
where we have used the posterior equation for the covariance, that the covariance is by assumption non-negative, and that, under the H\"older-continuity assumption, for a non-negative covariance  $\cov(l) \geq \max\{0, (1-cl^\alpha)\}$ and, hence,  
\begin{align*}
\cov^2(l) &\geq \cov(l) \max\{0, 1- cl^\alpha\} \\
&= 
\begin{cases} 
0 & \text{if } cl^\alpha \geq 1, \\
\cov(l) (1- cl^\alpha) & \text{otherwise,}
\end{cases} \\
 &\geq \begin{cases} 
0 & \text{if } cl^\alpha \geq 1, \\
1- 2 cl^\alpha + c^2l^{2\alpha} & \text{otherwise,}
\end{cases} \\
& \geq 1- 2 cl^\alpha.
\end{align*}
Combining \eqref{eq:cond_int_bound}, \eqref{eq:delta_t_i_bound} and \eqref{eq:variance_bound} we obtain  
\begin{align} 
&\int  (X_{t,i} - X_{t,u})^+ dP_{x_i,x_u}  \notag\\
& \leq  (2/\pi)^{1/2} c^{1/2} t^{\alpha/2} 
 \exp\left(- \frac{(\tilde \Delta_{i} - c((t-k)^\alpha +k^\alpha)
 (\tilde \Delta_{i} + \Delta))^2}{8 ct^{\alpha}} \right) \notag \\ 
 &=  (2/\pi)^{1/2} c^{1/2} t^{\alpha/2}\exp\left(- \frac{\tilde \Delta^2_{i}}{8 ct^{\alpha}} \Bigl(1 - c((t-k)^\alpha +k^\alpha)
 \Bigl(1 + \frac{\Delta}{\tilde \Delta_{i}}\Bigr)\Bigr)^2 \right). \label{eq:regretGPUp}
\end{align}
The bound is maximized for $t=m$. Substituting this bound back into \eqref{eq:marginal_integral} and after a few manipulations (see  Appendix \ref{sec:integration}) we see that
\begin{align}
\sum_{t=k+1}^m E\Bigl(&\max_{i \leq k} (X_{t,i} - X_{t,\pi_t})^+\Bigr) \notag \\ 
&\leq 
(m-k) \frac{a_m c^{1/2} k (k-1) }{8\pi (1-b_m)} 
\left(2\pi^{1/2}- (1-\Delta\sqrt{b_m/4}) 
\exp\left( - \frac{\Delta^2 b_m}{4}  \right) 
\right) \label{eq:regretPhaseII}
\end{align}
if $m>k$ and $\Delta < \sqrt{a_m}/(b_m\sqrt{2})$, 
where $a_m= 8cm^\alpha$ and $b_m= c((m-k)^\alpha + k^\alpha)$.

\paragraph{Combined Regret.} Combining \eqref{eq:regretPhaseI} with 
\eqref{eq:regretPhaseII} and using $\Delta_i \leq \Delta$ we can observe that the combined regret for any $m$ steps is  bounded by
\[
k(k-1) \left(\Delta + \sqrt{2} + (m-k)  
\frac{a_m c^{1/2}}{8\pi (1-b_m)} 
\left(2\pi^{1/2}- (1-\Delta\sqrt{b_m/4}) 
\exp\left( - \frac{\Delta^2 b_m}{4}  \right) 
\right)\right)
\]
if $m>k$ and $\Delta < \sqrt{a_m}/(b_m\sqrt{2})$.
 Given a time horizon $n$ we have 
$\lceil n/m \rceil$ many iterations of Phase I and II and the overall regret is bounded by 
\begin{align}
&(n/m + 1)k(k-1) \left(\Delta + \sqrt{2} +   
\frac{a_m c^{1/2}(m-k)}{8\pi (1-b_m)} 
\left(2\pi^{1/2}- (1-\Delta\sqrt{b_m/4}) 
\exp\left( - \frac{\Delta^2 b_m}{4}  \right) 
\right)\right) \notag \\
&\leq(n + m) k(k-1) \left(\frac{\Delta + \sqrt{2}}{m} +   
\frac{a_m c^{1/2}}{8\pi (1-b_m)} 
\left(2\pi^{1/2}- (1-\Delta\sqrt{b_m/4}) 
\exp\left( - \frac{\Delta^2 b_m}{4}  \right) 
\right)\right). \label{eq:reg_bound_app}
\end{align}
This is the regret bound stated in the proposition text. 

Instead of trying to find the $m$ that minimizes \eqref{eq:reg_bound_app} we optimize $m$ over the considerably simpler expression
\begin{equation} \label{eq:optM}
\frac{\Delta + \sqrt{2}}{m} +   
\frac{2c^{3/2}m^\alpha}{\sqrt{\pi}}\left(= \frac{\Delta + \sqrt{2}}{m} +   
\frac{a_m c^{1/2}}{4\sqrt{\pi}}\right).
\end{equation}
The $(1-b_m)$ term that we leave out is of minor relevance since $b_m$ is small and the negative term in the bracket that we leave out is not larger than $1$. By ignoring this latter term we lose only another constant. Minimizing  \eqref{eq:optM} with respect to $m$ and rounding up yields  
\[
m^\star = \left\lceil \left(\frac{\sqrt{\pi}(\Delta+  \sqrt{2})}{2\alpha c^{3/2}}\right)^{\frac{1}{1+\alpha}}  \right\rceil.
\] 
\subsubsection{Proof of Inequality \eqref{eq:marginal_integral}} \label{sec:conditioning}
Due to stationarity and since our policy depends only on the observations at the last sweep we have that  
$E(\max_{i \leq k} (X_{t',i} - X_{t',\pi_{t'}})^+) = E(\max_{i \leq k} (X_{t,i} - X_{t,\pi_t})^+)$ for 
any $t'$, $lm + k + 1 \leq  t' \leq lm$, $l \in \mathbb{N}$, and corresponding $t = t' -lm$. Now, consider any $t $,   $k + 1 \leq  t \leq m$ and, using $i^*$ for the choice of arm given the observations $X_{1,1}, \ldots, X_{k,k}$,  rewrite the regret in the following way 
\begin{align}
&E\Bigl(\max_{i \leq k} (X_{t,i} - X_{t,\pi_t})^+\Bigr)  \notag \\
&= \sum_{u=1}^k E\Bigl(\max_{i \leq k} \chi\{u = i^*\}\times (X_{t,i} - X_{t,i^*})^+\Bigr) \notag \\
&= \sum_{u=1}^k E\Bigl(\max_{i \not = u} \chi\{u = i^*\}\times (X_{t,i} - X_{t,u})^+\Bigr) \notag \\
&\leq \sum_{u=1}^k \sum_{i \not = u} E\Bigl( \chi\{u = i^*\}\times (X_{t,i} - X_{t,u})^+\Bigr)
 \notag \\
%&= \sum_{u=1}^k \sum_{i \not = u} EE\Bigl( \chi\{u = i^*\}\times %(X_{t,i} - X_{t,u})^+| X_{11},\ldots X_{kk}\Bigr) \notag \\
%&= \sum_{u=1}^k \sum_{i \not = u} E (\chi\{u = i^*\}\times E(  %(X_{t,i} - X_{t,u})^+| X_{11},\ldots, X_{kk})) \label{eq:cond1} \\
%&= \sum_{u=1}^k \sum_{i \not = u} E (\chi\{u = i^*\}\times E(  %(X_{t,i} - X_{t,u})^+| X_{ii}, X_{uu})) \label{eq:cond2} \\
&= \sum_{u=1}^k \sum_{i \not =  u} \int \int  \chi\{u = i^*\}\times (X_{t,i} - X_{t,u})^+ dP_{x_1,\ldots, x_k} d\nu(x_1,\ldots,x_k) 
\notag \\
&= \sum_{u=1}^k \sum_{i \not =  u} \int  P_{x_1,\ldots,x_k}( u = i^*)  \times \int  (X_{t,i} - X_{t,u})^+ dP_{x_1,\ldots, x_k} d\nu(x_1,\ldots,x_k) \label{eq:cond1} \\
&= \sum_{u=1}^k \sum_{i \not =  u} \int  P_{x_1,\ldots,x_k}( u = i^*)  \times \int  (X_{t,i} - X_{t,u})^+ dP_{x_i, x_u} d\nu(x_1,\ldots,x_k) \label{eq:cond2} 
\end{align}
where \eqref{eq:cond1} follows because $\chi\{u = i^*\}$ is independent of $(X_{t,i} - X_{t,u})^+$ given $X_{11}=x_1,\ldots, X_{kk} = x_k$ and \eqref{eq:cond2} follows since $(X_{t,i} - X_{t,u})^+$ only depends on $X_{i,i}$ and $X_{u,u}$.

\subsubsection{Bound on $\varepsilon(t)$.} \label{sec:varepsilon}
The term $\varepsilon(t)$ that we introduced is equal to
\begin{align*}
\varepsilon(t) = &\mu_{u} (\cov(t-k) -\cov(t - u)) - \mu_i (\cov(t-k) -\cov(t- i)) \\
&+ x_{u} (\cov(t-u)-\cov(t-k)) - x_i (\cov(t - i) - \cov(t-k)).
\end{align*}
Since, due to our H\"older assumption, $|\cov(t-k) -\cov(t  - u)| \leq c (k-u)^\alpha \leq ck^\alpha$ and $|\cov(t-k) -\cov(t - i)| \leq ck^\alpha$, we have the following bound. 
\[
|\varepsilon(t)| \leq (\Delta + \tilde \Delta_{i}) ck^\alpha. 
\]

\subsubsection{Proof of Inequality \eqref{eq:regretPhaseII}} \label{sec:integration} 
Denote in the following $X_u := X_{u,u}$ and $X_i = X_{i,i}$ and 
recall that $\tilde \Delta_i = X_u - X_i$ is normally distributed with mean $\mu_u - \mu_i$ and variance $2$. %The multiplication with $\chi\{u=i^*\}$ implies, in particular, that we need to integrate $\tilde \Delta_i$ only over the positive real line. 
Writing 
\[
f(X_u,X_i)  = (2/\pi)^{1/2} c^{1/2} m^{\alpha/2} \exp\left(- \frac{\tilde \Delta^2_{i}}{8 cm^{\alpha}} \Bigl(1 - c((m-k)^\alpha +k^\alpha)
 \Bigl(1 + \frac{\Delta}{\tilde \Delta_{i}}\Bigr)\Bigr)^2 \right)
\]
we have 
\begin{align}
&\int  P_{x_1,\ldots,x_k}(u = i^*)  \times \int  (X_{t,i} - X_{t,u})^+ dP_{x_i,x_u} \, d\nu(x_1,\ldots,x_k)  \notag \\
&\leq \int  P_{x_1,\ldots,x_k}(X_u \geq X_i)  \times \int  (X_{t,i} - X_{t,u})^+ dP_{x_i,x_u} \, d\nu(x_1,\ldots,x_k) \notag \\
&\leq \int  \chi\{x_u \geq x_i\} \times f(x_u,x_i)  \, d\nu(x_1,\ldots,x_k) \label{eq:condbound1} \\
&= \int  \chi\{x_u - x_i \geq 0\} \times f(x_u,x_i) \, d\nu(x_u,x_i), \label{eq:condbound2}
\end{align}
where the  inequality in \eqref{eq:condbound1} follows because  $\{X_u \geq X_i\}$ is only dependent on $x_u$ and $x_i$ and by bounding the inner integral with  \eqref{eq:regretGPUp}.
Multiplying the density of $\tilde \Delta_i$ by $f$ and using 
$a= 8cm^\alpha$, $b= c((m-k)^\alpha + k^\alpha)$, $d= (ac)^{1/2}/(4\pi)$,  
in the case where $\Delta  < \sqrt{a}/(b\sqrt{2})$, we obtain
\begin{align*}
\int  \chi\{x_u - x_i \geq 0\} &\times f(x_u,x_i) \, d\nu(x_u,x_i)\\
&=d \int_{0}^\infty \exp\left( - \frac{(\tilde \Delta_i (1 - b) - \Delta b )^2}{a} - \frac{(\tilde \Delta_i - (\mu_u - \mu_i))^2}{4}  \right) d \tilde \Delta_i  \\
%\iffalse
%&=  d \exp\left( - \left(\tilde \Delta_i^2 \frac{(1 - b)^2 + (a/4) }{a}
%- 2 \tilde \Delta_i ( (1-b) \Delta (b/a)  + \Delta_{u,i}/4) + \Delta^2 %b^2/a + \Delta_{u,i}^2/4
%\right) \right) \\
%&\leq d \exp\left( - \left(\tilde \Delta_i^2 \frac{(1 - b)^2}{a}
%- 2 \tilde \Delta_i ( (1-b) \Delta (b/a)  + \Delta_{u,i}/4) + \Delta^2 %b^2/a + \Delta_{u,i}^2/4
%\right) \right) \\
%\fi
&\leq d \int_{0}^\infty \exp\left( - \frac{(\tilde \Delta_i (1 - b) - \Delta b )^2}{a}  \right) d \tilde \Delta_i \\
&= \frac{d}{1-b} \int_{-\Delta b}^\infty \exp\left( - \frac{\tilde \Delta_i^2}{a}  \right) d \tilde \Delta_i \\
%&= \frac{\sqrt{a} d }{(1-b)\sqrt{2}} \int_{\frac{-\sqrt{2}\Delta b}{\sqrt{a}}}^\infty \exp\left( - \frac{\tilde \Delta_i^2}{2}  \right) d \tilde \Delta_i \\
&= \frac{\sqrt{a} d }{(1-b)\sqrt{2}} \int_{-\infty}^{\frac{\sqrt{2}\Delta b}{\sqrt{a}}} \exp\left( - \frac{\tilde \Delta_i^2}{2}  \right) d \tilde \Delta_i \\
&=\frac{\sqrt{a} d \sqrt{2\pi} }{(1-b)\sqrt{2}} \left(1- \frac{1}{\sqrt{2\pi}} \int_{\frac{\sqrt{2}\Delta b}{\sqrt{a}}}^\infty \exp\left( - \frac{\tilde \Delta_i^2}{2}  \right) d \tilde \Delta_i \right)\\
&\leq \frac{\sqrt{\pi a} d  }{(1-b)} \left(1- \frac{1-\sqrt{2}\Delta b/\sqrt{a}}{2\sqrt{2\pi} } \exp\left( - \frac{\Delta^2 b^2}{a}  \right) \right) \\ 
%&\leq \frac{\sqrt{\pi a} d  }{(1-b)} \left(1- \frac{\sqrt{8}}{2 \sqrt{2\pi} \sqrt{2b}\Delta } \exp\left( - \frac{\Delta^2 b}{8}  \right) \right) \\
%&\leq \frac{\sqrt{ a}  }{8 \pi (1-b)} \left(\sqrt{8\pi} - \left(1- 
%\frac{\sqrt{2}b \Delta}{\sqrt{a}} \right)  \exp\left( - \frac{b\Delta^2}{8}  \right) \right)
&= \frac{a c^{1/2}  }{4(1-b)\sqrt{\pi}} \left(1- \frac{1-\sqrt{2}\Delta b/\sqrt{a}}{2\sqrt{2\pi} } \exp\left( - \frac{\Delta^2 b^2}{a}  \right) \right) \\
&\leq \frac{a c^{1/2}  }{8\pi (1-b)} \left(2\pi^{1/2}- (1-\Delta b/\sqrt{a}) \exp\left( - \frac{\Delta^2 b^2}{a}  \right) \right) 
\end{align*}
and for $m > k$, since then $a/4 \geq b$, this can be further bounded by   
\[
\frac{a c^{1/2}  }{8\pi (1-b)} \left(2\pi^{1/2}- (1-\Delta\sqrt{b/4}) \exp\left( - \frac{\Delta^2 b}{4}  \right) \right).
\]

\newpage
%\bibliography{mixing_bandits.bib}

\end{document}